\documentclass[12pt,a4paper]{amsart}

\usepackage[english]{babel}
\usepackage[utf8x]{inputenc}
\usepackage{amsfonts}
\usepackage{amssymb}
\usepackage{amsmath}
\usepackage{amsrefs}
\usepackage{mathrsfs}
\usepackage{graphicx}
\usepackage{framed}
\usepackage{mdframed}
\usepackage{color}

\usepackage{dsfont}
\usepackage{stmaryrd}

\usepackage[scr=boondox]{mathalfa}

\newtheorem{theorem}{Theorem}[section]
\newtheorem{lemma}[theorem]{Lemma}
\newtheorem{proposition}[theorem]{Proposition}

\newtheorem{corollary}[theorem]{Corollary}

\theoremstyle{definition}
\newtheorem{definition}[theorem]{Definition}

\newcommand{\ad}{\mathrm{ad}}
\newcommand{\Ad}{\mathrm{Ad}}
\DeclareMathOperator{\Hol}{\mathrm{Hol}}
\DeclareMathOperator{\Aut}{\mathrm{Aut}}

\DeclareMathOperator{\GLin}{\mathrm{GL}}

\DeclareMathOperator{\PGL}{\mathrm{PGL}}

\newcommand{\MC}[1]{\omega_{{}_{#1}}}

\newcommand{\Rt}[1]{\operatorname{R}_{#1}}

\usepackage[all]{xy}

\usepackage[normalem]{ulem}

\usepackage[T2A,OT1]{fontenc}

\DeclareFontEncoding{T3}{}{}
\DeclareFontSubstitution{T3}{cmr}{m}{n}
\DeclareSymbolFont{tipa}{T3}{cmr}{m}{sl}
\SetSymbolFont{tipa}{bold}{T3}{cmr}{bx}{sl}

\DeclareMathSymbol{\kgf}{\mathord}{tipa}{'255}

\title[Curvature trees and essential automorphisms]{Higher rank parabolic geometries with essential automorphisms and nonvanishing curvature}
\author{Jacob W. Erickson}
\thanks{Partially supported by the Brin Graduate Fellowship at the University of Maryland}
\date{\today}

\begin{document}
\maketitle

\begin{abstract}We construct infinite families of regular normal Cartan geometries with nonvanishing curvature and essential automorphisms on closed manifolds for many higher rank parabolic model geometries. To do this, we use particular elements of the kernel of the Kostant Laplacian to construct homogeneous Cartan geometries of the desired type, giving a global realization of an elegant local construction due to Kruglikov and The, and then modify these homogeneous geometries to make their base manifolds compact. As a demonstration, we apply the construction to quaternionic contact structrures of mixed signature, among other examples.\end{abstract}

\section{Introduction}
Given a conformal class $[\mathrm{g}]$ of Riemannian metrics on a smooth manifold $M$, the group of conformal transformations of $(M,[\mathrm{g}])$ is said to be \emph{inessential} if it acts isometrically for some $\mathrm{g}'\in[\mathrm{g}]$, and \emph{essential} otherwise. In other words, essential groups of conformal transformations are those which only preserve the conformal structure itself, not any of the underlying Riemannian metrics.

Conformal structures with essential groups of conformal transformations are remarkably atypical. Indeed, it is a result of Ferrand \cite{Ferrand} and Obata \cite{Obata}, verifying a conjecture of Andr\'{e} Lichnerowicz, that if $M$ is compact and the group of conformal transformations is essential, then $M$ is necessarily the sphere and $[\mathrm{g}]$ is the conformal class of the round metric. An analogous result for strictly pseudoconvex CR structures was later obtained by Webster \cite{Webster} and Schoen \cite{Schoen}: if a compact strictly pseudoconvex CR manifold $M$ has noncompact CR automorphism group, hence if it does not preserve an underlying pseudo-hermitian form for the CR structure, then $M$ is CR-equivalent to the sphere with the standard CR structure.

Both conformal structures of Riemannian signature and strictly pseudoconvex CR structures are examples of parabolic geometries, and in both cases, the underlying model geometries $(G,P)$ for these structures have $G$ with real rank 1. This led Frances to a wondrous generalization of all of the above results in \cite{RankOneFO}: if $(\mathscr{G},\omega)$ is a regular normal Cartan geometry of type $(G,P)$ over a compact smooth manifold $M$, with $G$ simple of real rank 1 and $P$ parabolic, then the automorphism group of $(\mathscr{G},\omega)$ acts nonproperly on $M$ if and only if $(\mathscr{G},\omega)$ is isomorphic to the Klein geometry of type $(G,P)$. By a later result of \cite{Alt}, an automorphism group that is essential in the sense of Definition \ref{ess} below necessarily acts nonproperly.

While there are immediate topological obstructions to extending this rigidity result to parabolic geometries of real rank greater than 1, it was reasonable to conjecture, as Frances did in \cite{Frances}, that parabolic geometries with essential\footnote{In \cite{Frances}, Frances actually uses a slightly weaker definition of essential for general parabolic geometries that does not coincide with the usual notion when applied to, for example, strictly pseudoconvex CR structures. We will be using the definition from \cite{Alt} instead, but our counterexamples will be valid for both definitions.} automorphism groups would necessarily be flat, meaning that they would still be \textit{locally} isomorphic to their Klein geometries.

However, Frances later showed in \cite{FrancesCounterexample} that even this weakened conjecture does not hold in higher rank cases. Specifically, he gave a construction of an infinite family of conformal structures of signature $(p,q)$, with $\min(p,q)>1$, on $S^1\times S^{p+q-1}$ with essential automorphism group but nonvanishing curvature. Soon afterward, Case, Curry, and Matveev gave an analogous construction in \cite{MixedCRstructures} for CR structures of signature $(p,q)$, with $\min(p,q)>1$. In both cases, the corresponding parabolic model geometries had real rank $\min(p,q)+1>2$. The results of this current paper will generalize both of these constructions to other parabolic geometries of real rank greater than 2.

To do this for Cartan geometries of type $(G,P)$, we first construct non-flat homogeneous Cartan geometries of type $(G,P)$. In fact, we choose the curvature that we want and then construct a homogeneous geometry with that curvature, giving a global realization of an elegant local construction from \cite{GapPhenomenon}. Originally, the author constructed these homogeneous geometries by a completely different, computationally intensive method, where we would start with a ``seed" of curvature, this seed would produce (restricted) roots for the Lie algebra $\mathfrak{g}$, and we would then ``grow" a deformation of the flat homogeneous geometry on $G_-G_0\times_{G_0}P$ from these roots. As such, we chose to call these homogeneous Cartan geometries ``curvature trees", and we detail their construction in Section \ref{forest}.

Under relatively mild conditions, we then construct non-flat regular normal Cartan geometries with essential automorphisms over a closed base manifold using these curvature trees in Section \ref{spotlight}. Several examples detailing how the construction applies in particular cases are given in Section \ref{examples}.

We should note that the requirement that the real rank be greater than 2 is necessary for us to construct compact examples. As such, our construction does not give compact examples in the conformal Lorentzian case, for which the question of whether essentiality implies flatness is still, at the time of this writing, an active area of research. Indeed, a recent result from \cite{FrancesMelnick} shows that, for real-analytic conformal Lorentzian structures, essentiality does imply flatness in dimension 3.

\section*{Acknowledgements}The author is especially grateful to Karin Melnick for several helpful discussions, as well as for catching a mistake in an earlier version of this paper.

\section{Preliminaries}
Throughout this paper, we will assume some familiarity with parabolic geometries and Cartan geometries in general. The author recommends \cite{CapSlovak} for a thorough introduction to the topic of parabolic geometries. For clarity, we will briefly review some of the relevant definitions and specify some notation.

\subsection{Parabolic model geometries}
\begin{definition}A \emph{model geometry} is a pair $(G,H)$, where $G$ is a Lie group and $H$ is a closed subgroup such that $G/H$ is connected. A model geometry $(G,P)$ is said to be \emph{parabolic} when $G$ is a semisimple Lie group and $P$ is a parabolic subgroup.\end{definition}

For a given parabolic model geometry $(G,P)$, we get an $\Ad_P$-invariant filtration \[\mathfrak{g}=\mathfrak{g}^{-k}\supset\mathfrak{g}^{-k+1}\supset\cdots\supset\mathfrak{g}^k\supset\{0\}\] of $\mathfrak{g}$ for some $k$, with $\mathfrak{g}^0=\mathfrak{p}$ and $\mathfrak{g}^1=\mathfrak{p}_+$ the nilradical of $\mathfrak{p}$. For this canonical filtration, we can choose an associated $|k|$-grading $\mathfrak{g}_{-k}\oplus\cdots\oplus\mathfrak{g}_k$ of $\mathfrak{g}$, with $\mathfrak{g}^i=\bigoplus_{\ell\geq i}\mathfrak{g}_\ell$.

The subgroup \[G_0:=\{R\in P:\Ad_R(\mathfrak{g}_i)\subseteq\mathfrak{g}_i\text{ for each}-k\leq i\leq k\}\] with Lie algebra $\mathfrak{g}_0$ is called the \emph{Levi subgroup} of $P$. This closed subgroup of $P$ is reductive, so that its Lie algebra splits as $\mathfrak{z}(\mathfrak{g}_0)\oplus\mathfrak{g}_0^\text{ss}$ with $\mathfrak{g}_0^\text{ss}$ semisimple. In the center of this Lie algebra $\mathfrak{g}_0$, there is a unique element $E_\text{gr}\in\mathfrak{z}(\mathfrak{g}_0)$, called the \emph{grading element}, such that $[E_\text{gr},X]=iX$ whenever $X\in\mathfrak{g}_i$. We also get a subalgebra $\mathfrak{g}_-$ generated by all of the negative grading components, dual to the nilradical $\mathfrak{p}_+$ of $\mathfrak{p}$, and a corresponding closed subgroup $G_-$ of $G$.

Typically, we will want to work with restricted roots of $\mathfrak{g}$, so we use the following definition to make sure the restricted roots we use are compatible with our choice of parabolic subgroup.

\begin{definition}For a given parabolic model geometry $(G,P)$, we say that $(\theta,\mathfrak{c},\mathfrak{a},\Delta,\Delta^+)$ is a \emph{compatible restricted root system} if $\theta$ is a Cartan involution of $\mathfrak{g}$, $\mathfrak{c}$ is a maximally noncompact $\theta$-stable Cartan subalgebra containing the grading element $E_\text{gr}\in\mathfrak{g}_0$, $\mathfrak{a}$ is the maximal subalgebra of $\mathfrak{c}$ such that $\theta|_\mathfrak{a}=-\mathrm{id}_\mathfrak{a}$, $\Delta\subset\mathfrak{a}^\vee$ is the set of restricted roots for $\mathfrak{g}$ with respect to $\mathfrak{a}$, and $\Delta^+\subset\Delta$ is a choice of positive roots such that $\Delta^+(\mathfrak{p}_+):=\{\alpha\in\Delta:\alpha(E_\text{gr})>0\}\subseteq\Delta^+$.\end{definition}

We let $\kgf$ denote the Killing form on $\mathfrak{g}$, and for $X\in\mathfrak{g}$, we denote by $X_\kgf\in\mathfrak{g}^\vee$ the element of the dual of $\mathfrak{g}$ given by $X_\kgf(Y):=\kgf(X,Y)$. Because $G$ is semisimple for parabolic model geometries $(G,P)$, $\kgf$ is nondegenerate, so the map $X\mapsto X_\kgf$ is a linear isomorphism from $\mathfrak{g}$ to $\mathfrak{g}^\vee$. For $\psi\in\mathfrak{g}^\vee$, we denote by $\psi^\kgf\in\mathfrak{g}$ the image of $\psi$ under the inverse of this isomorphism, so that $(\psi^\kgf)_\kgf=\psi$.

For elements $\nu\in\mathfrak{a}^\vee$, we will identify $\nu$ with the linear functional $\tilde{\nu}\in\mathfrak{g}^\vee$ such that $\tilde{\nu}|_\mathfrak{a}=\nu$ and $\tilde{\nu}(X)=0$ for every $X$ such that $\kgf(X,Y)=0$ for all $Y\in\mathfrak{a}$, and we will then write $\nu^\kgf\in\mathfrak{a}<\mathfrak{g}$ for $\tilde{\nu}^\kgf$. Using this, we get an inner product on $\mathfrak{a}^\vee$, which we also denote by $\kgf$, given by $\kgf(\nu_1,\nu_2):=\kgf(\nu_1^\kgf,\nu_2^\kgf)$.

For $(\theta,\mathfrak{c},\mathfrak{a},\Delta,\Delta^+)$ a compatible restricted root system, we also get a symmetric negative-definite form $\kgf_\theta$ given by $\kgf_\theta(X,Y):=\kgf(X,\theta(Y))$.

\subsection{General Cartan geometries}
The bulk of the geometric machinery of this paper is handled through Cartan geometries.

\begin{definition}Given a model geometry $(G,H)$ and a smooth manifold $M$, a \emph{Cartan geometry of type $(G,H)$ over $M$} is a pair $(\mathscr{G},\omega)$, where $\mathscr{G}$ is a principal $H$-bundle over $M$ and $\omega$ is a $\mathfrak{g}$-valued one-form satisfying the following:
\begin{itemize}
\item For every $\mathscr{g}\in\mathscr{G}$, $\omega_\mathscr{g}:T_\mathscr{g}\mathscr{G}\to\mathfrak{g}$ is a linear isomorphism;
\item For every $h\in H$, $\Rt{h}^*\omega=\Ad_{h^{-1}}\omega$\,;
\item For every $Y\in\mathfrak{h}$ and $\mathscr{g}\in\mathscr{G}$, the flow of the vector field $\omega^{-1}(Y)$ is given by $\exp(t\omega^{-1}(Y))\mathscr{g}=\mathscr{g}\exp(tY)$.\end{itemize}\end{definition}

The Cartan connection $\omega$ of a Cartan geometry $(\mathscr{G},\omega)$ of type $(G,H)$ behaves similarly to the Maurer-Cartan form $\MC{G}$ on $G$, which encodes the diffeo-geometric structure of the model homogeneous geometry on $G/H$ in the sense of Klein's Erlangen program. Indeed, the geometric structure of the model geometry can be encoded as a particular type of Cartan geometry called the \emph{Klein geometry}.

\begin{definition}The \emph{Klein geometry} of type $(G,H)$ is the Cartan geometry $(G,\MC{G})$ of type $(G,H)$ over $G/H$, where $\MC{G}$ is the Maurer-Cartan form on $G$.\end{definition}

The \emph{curvature} of a Cartan geometry then tells us how it locally differs from the model geometry.

\begin{definition}For a Cartan geometry $(\mathscr{G},\omega)$ of type $(G,H)$, its \emph{curvature} $\Omega$ is given by $\Omega=\mathrm{d}\omega+\frac{1}{2}[\omega,\omega]$. For $X,Y\in\mathfrak{g}$, the curvature applied to $\omega^{-1}(X)\wedge\omega^{-1}(Y)$ is given by \begin{align*}\Omega^\omega(X\wedge Y) & :=\Omega(\omega^{-1}(X)\wedge\omega^{-1}(Y)) \\ & =\mathrm{d}\omega(\omega^{-1}(X)\wedge\omega^{-1}(Y))+[X,Y] \\ & =\omega^{-1}(X)\omega(\omega^{-1}(Y))-\omega^{-1}(Y)\omega(\omega^{-1}(X)) \\ & \hspace{5em} -\omega([\omega^{-1}(X),\omega^{-1}(Y)])+[X,Y] \\ & =[X,Y]-\omega([\omega^{-1}(X),\omega^{-1}(Y)]).\end{align*} We say that a Cartan geometry is \emph{flat} when its curvature vanishes identically.\end{definition}

We also get a notion of morphisms and automorphisms.

\begin{definition}Given two Cartan geometries $(\mathscr{G}_1,\omega_1)$ and $(\mathscr{G}_2,\omega_2)$ of type $(G,H)$, a \emph{geometric map} is an $H$-equivariant map $\varphi:\mathscr{G}_1\to\mathscr{G}_2$ such that $\varphi^*\omega_2=\omega_1$. A geometric map that is also a diffeomorphism is called a \emph{(geometric) isomorphism}, and a geometric isomorphism from a Cartan geometry to itself is called a \emph{(geometric) automorphism}.\end{definition}

The set of all automorphisms of a Cartan geometry $(\mathscr{G},\omega)$ naturally inherits the structure of a Lie group, which we denote by $\Aut(\mathscr{G},\omega)$.

To compare Cartan geometries modeled on different homogeneous geometries, we use \emph{extension functors}.

\begin{definition}Given two model geometries $(Q,K)$ and $(G,H)$, suppose we have a pair $(i,\psi)$, where $i:K\to H$ is a homomorphism of Lie groups and $\psi:\mathfrak{q}\to\mathfrak{g}$ is a linear map, such that the following conditions are satisfied.
\begin{itemize}
\item For all $k\in K$, $\psi\circ\Ad_k=\Ad_{i(k)}\circ\psi$;
\item $\psi|_\mathfrak{k}=i_*$;
\item The induced map $\bar{\psi}:\mathfrak{q}/\mathfrak{k}\to\mathfrak{g}/\mathfrak{h}$ is a linear isomorphism.
\end{itemize}
Then, the \emph{extension functor} induced by $(i,\psi)$ takes Cartan geometries $(\mathscr{Q},\upsilon)$ of type $(Q,K)$ and outputs Cartan geometries of type $(G,H)$ given by $(\mathscr{Q}\times_i H,\omega)$, where for each $(\mathscr{q},h)\in\mathscr{Q}\times_i H$, $\omega_{(\mathscr{q},h)}:=\Ad_{h^{-1}}\circ\psi(\upsilon_\mathscr{q})+\MC{H}$.\end{definition}

Given an extension functor from type $(Q,K)$ to type $(G,H)$ with $i:K\to H$ injective, we get a corresponding \emph{homogeneous Cartan geometry} of type $(G,H)$ given by the output of the extension functor applied to the Klein geometry of type $(Q,K)$. Homogeneous Cartan geometries are precisely the Cartan geometries $(\mathscr{G},\omega)$ for which the induced action of $\Aut(\mathscr{G},\omega)$ on the base manifold is transitive.

\subsection{Curvature restrictions for parabolic geometries}
For Cartan geometries of type $(G,P)$, where $(G,P)$ is a parabolic model geometry, it is often convenient to place certain restrictions on their curvatures. The first of these, \emph{regularity}, is fairly easy to understand geometrically.

\begin{definition}A Cartan geometry $(\mathscr{G},\omega)$ of type $(G,P)$ is \emph{regular} if and only if the curvature $\Omega$ satisfies $\Omega^\omega(X\wedge Y)\in\mathfrak{g}^{i+j+1}$ whenever $X\in\mathfrak{g}^i$ and $Y\in\mathfrak{g}^j$.\end{definition}

Regularity tells us that, while $(\mathscr{G},\omega)$ is allowed to have torsion, that torsion cannot ``escape" the filtration on $\mathfrak{g}$ induced by $\mathfrak{p}$. In particular, if $q_{{}_{P*}}(X)$ and $q_{{}_{P*}}(Y)$ are contained in the distribution on the base manifold $M$ given by taking the image of $\omega^{-1}(\mathfrak{g}^{-1})$ under the quotient map $q_{{}_P}:\mathscr{G}\to M$, then the curvature applied to $X\wedge Y$ must remain in $\mathfrak{g}^{-1}$.


The second condition, \emph{normality}, is slightly harder to understand geometrically, and the author still considers it somewhat mysterious from that perspective. The \emph{Kostant codifferential} $\partial^*:\Lambda^k(\mathfrak{g}/\mathfrak{p})^\vee\otimes\mathfrak{g}\to\Lambda^{k-1}(\mathfrak{g}/\mathfrak{p})^\vee\otimes\mathfrak{g}$ is defined, for $Y_1,\dots,Y_k\in\mathfrak{p}_+$ and $X\in\mathfrak{g}$, by \begin{align*}\partial^*((Y_1)_\kgf\wedge\cdots\wedge(Y_k)_\kgf\otimes X):= & \sum_{i=1}^k(-1)^i\left(\wedge_{\ell\neq i}(Y_\ell)_\kgf\right)\otimes[Y_i,X] \\ & +\sum_{j>i}(-1)^{i+j}[Y_i,Y_j]_\kgf\wedge\left(\wedge_{\ell\not\in\{i,j\}}(Y_\ell)_\kgf\right)\otimes X,\end{align*} and a Cartan geometry is normal if and only if the Kostant codifferential vanishes on the curvature form.

\begin{definition}A Cartan geometry $(\mathscr{G},\omega)$ of type $(G,P)$ is \emph{normal} if and only if its curvature $\Omega$ satisfies $\partial^*(\Omega^\omega)=0$.\end{definition}

Under the identifications $(\mathfrak{g}/\mathfrak{p})^\vee\approx\mathfrak{g}_-^\vee\approx\mathfrak{p}_+$, $\partial^*$ is precisely the boundary operator for computing the homology groups $H_k(\mathfrak{p}_+;\mathfrak{g})$. Thus, if $(\mathscr{G},\omega)$ is normal, then the curvature determines a section of $\mathscr{G}\times_P H_2(\mathfrak{p}_+;\mathfrak{g})$ by taking the image of $\Omega^\omega$ under the map $\mathscr{G}\times_P\ker(\partial^*)\to\mathscr{G}\times_P H_2(\mathfrak{p}_+;\mathfrak{g})$. This section is called the \emph{harmonic curvature} of $(\mathscr{G},\omega)$.

We may also define the usual differential $\partial:\Lambda^k\mathfrak{g}_-^\vee\otimes\mathfrak{g}\to\Lambda^{k+1}\mathfrak{g}_-^\vee\otimes\mathfrak{g}$ by \begin{align*}\partial\alpha(X_0\wedge\cdots\wedge X_k):= & \sum_{i=0}^k(-1)^i[X_i,\alpha(\wedge_{\ell\neq i}X_\ell)] \\ & +\sum_{j>i}(-1)^{i+j}\alpha([X_i,X_j]\wedge\left(\wedge_{\ell\not\in\{i,j\}}X_\ell\right)),\end{align*} where $X_0,\dots,X_k\in\mathfrak{g}_-$ and $\alpha\in\Lambda^k\mathfrak{g}_-^\vee\otimes\mathfrak{g}$. It turns out (see, for example, Theorem 3.3.1 of \cite{CapSlovak}) that, using the \emph{Kostant Laplacian} \[\square:=\partial^*\circ\partial+\partial\circ\partial^*,\] we get a Hodge decomposition of $G_0$-representations \[\Lambda^k(\mathfrak{g}/\mathfrak{p})^\vee\otimes\mathfrak{g}=\mathrm{im}(\partial)\oplus\ker(\square)\oplus\mathrm{im}(\partial^*),\] with $\ker(\partial)=\mathrm{im}(\partial)\oplus\ker(\square)$ and $\ker(\partial^*)=\mathrm{im}(\partial^*)\oplus\ker(\square)$. In particular, as $G_0$-representations, \[H^k(\mathfrak{g}_-;\mathfrak{g})\approx\ker(\square)\approx H_k(\mathfrak{p}_+;\mathfrak{g}),\] and the harmonic curvature corresponds to the Kostant-harmonic component of $\Omega^\omega$.

For $(\mathscr{G},\omega)$ regular and normal, the Kostant-harmonic component of $\Omega^\omega$ must be in the $G_0$-subrepresentation $\ker(\square)_+$ consisting of elements of $\ker(\square)$ that map $X\wedge Y$ into $\mathfrak{g}^{i+j+1}$ whenever $X\in\mathfrak{g}^i$ and $Y\in\mathfrak{g}^j$. We can equivalently define $\ker(\square)_+$ to be the sum of the positive eigenspaces of $E_\text{gr}$ acting on $\Lambda^2(\mathfrak{g}/\mathfrak{p})^\vee\otimes\mathfrak{g}$.

As a consequence of Kostant's version of the Bott-Borel-Weil theorem (see, for example, Theorem 3.3.5 in \cite{CapSlovak}), if $G$ is split-real and simple, then the $G_0$-irreducible components of $\ker(\square)$ will have lowest weight vectors of the form \[(\eta_{\alpha_i})_\kgf\wedge(\eta_{s_{\alpha_i}(\alpha_j)})_\kgf\otimes\eta_{-s_{\alpha_i}s_{\alpha_j}(\mu)},\] where $\alpha_i$ and $\alpha_j$ are simple roots, $\mu$ is the highest root of the adjoint representation, $s_\alpha:\nu\mapsto\nu-2\frac{\kgf(\alpha,\nu)}{\kgf(\alpha,\alpha)}\alpha$, and $\eta_\nu\in\mathfrak{g}_\nu$ for $\nu\in\Delta$. Moreover, the lowest weight vectors of this form that are contained in $\ker(\square)_+$ were, under the appropriate identifications, listed in \cite{Yamaguchi} (with some small corrections in \cite{Yamaguchi2}), so if we want to find a lowest weight vector of this form, then we can just look one up on Yamaguchi's list.

More generally, for arbitrary parabolic $(G,P)$, we will get lowest weight vectors of $\ker(\square)$ of the form \[\sum_k(\eta_{\beta,k})_\kgf\wedge(\eta_{\gamma,k})_\kgf\otimes\eta_{\zeta,k}\] for some $\beta,\gamma\in\Delta^+(\mathfrak{p}_+)$ and $\zeta\in\Delta$. For the purposes of this paper, we will be specifically interested in such elements contained in $\ker(\square)_+$ with $\zeta(E_\text{gr})\leq 0$, which we can describe in the following lemma.

\begin{lemma}\label{assume} Suppose $\Omega=\sum_k(\eta_{\beta,k})_\kgf\wedge(\eta_{\gamma,k})_\kgf\otimes\eta_{\zeta,k}\in\ker(\square)_+$ is a lowest weight vector, where $\beta,\gamma\in\Delta^+(\mathfrak{p}_+)$ and $\zeta\in\Delta\setminus\Delta^+(\mathfrak{p}_+)$ such that $\beta-\gamma\not\in\Delta^+$, with $\eta_{\beta,k}\in\mathfrak{g}_\beta$, $\eta_{\gamma,k}\in\mathfrak{g}_\gamma$, and $\eta_{\zeta,k}\in\mathfrak{g}_\zeta$ for each $k$. Then, $\zeta\in-\Delta^+$ and $\beta$ is a simple restricted root.\end{lemma}
\begin{proof}Without loss of generality, we may assume that the $\eta_{\beta,k}\wedge\eta_{\gamma,k}$, indexed over $k$, are linearly independent, since otherwise we can just rewrite $\Omega$ by combining linearly dependent terms until the $\eta_{\beta,k}\wedge\eta_{\gamma,k}$ are linearly independent. Similarly, we may assume that the union of the set of all the $\eta_{\beta,k}\otimes\eta_{\zeta,k}$ with the set of all the $\eta_{\gamma,k}\otimes\eta_{\zeta,k}$ is linearly independent by combining linearly dependent terms if necessary.

If we had $\zeta\in\Delta^+(\mathfrak{g}_0)$, then because $\Omega$ is a lowest weight vector, we would have \begin{align*}0 & =\theta(\eta_{\zeta,k})\cdot\Omega \\ & =\sum_k\Big([\theta(\eta_{\zeta,k}),\eta_{\beta,k}]_\kgf\wedge(\eta_{\gamma,k})_\kgf\otimes\eta_{\zeta,k}+(\eta_{\beta,k})_\kgf\wedge[\theta(\eta_{\zeta,k}),\eta_{\gamma,k}]_\kgf\otimes\eta_{\zeta,k} \\ & \hspace{5em} +(\eta_{\beta,k})_\kgf\wedge(\eta_{\gamma,k})_\kgf\otimes[\theta(\eta_{\zeta,k}),\eta_{\zeta,k}]\Big).\end{align*} Because the $\eta_{\beta,k}\wedge\eta_{\gamma,k}$ are linearly independent, this would require that $[\theta(\eta_{\zeta,k}),\eta_{\zeta,k}]=0$, which cannot happen unless $\eta_{\zeta,k}=0$. In particular, either $\zeta\in-\Delta^+(\mathfrak{p}_+)$ or $\zeta\in-\Delta^+(\mathfrak{g}_0)$, hence $\zeta\in-\Delta^+$.

Because $\Omega\in\ker(\square)_+$, we know that $\partial\Omega=0$. Thus, for $X,Y\in\mathfrak{g}_-$ such that $\iota_X\Omega=\iota_Y\Omega=0$, $\partial\Omega(X\wedge Y\wedge Z)=\Omega([X,Y]\wedge Z)=0$, so $\iota_{[X,Y]}\Omega=0$. Since $\mathfrak{g}_{-i-1}=[\mathfrak{g}_{-1},\mathfrak{g}_{-i}]$, it follows that $\Omega=0$ unless $\beta(E_\text{gr})=1$. If $\beta=\nu_1+\nu_2$ for $\nu_1,\nu_2\in\Delta^+$ with $\nu_1(E_\text{gr})\geq\nu_2(E_\text{gr})$, then $\nu_1(E_\text{gr})=1$ and $\nu_2(E_\text{gr})=0$. In this case, we would have $\beta-\nu_2=\nu_1$, but for every $\eta_{-\nu_2}\in\mathfrak{g}_{-\nu_2}$, \begin{align*}0 & =\eta_{-\nu_2}\cdot\Omega \\ & =\sum_k\Big([\eta_{-\nu_2},\eta_{\beta,k}]_\kgf\wedge(\eta_{\gamma,k})_\kgf\otimes\eta_{\zeta,k}+(\eta_{\beta,k})_\kgf\wedge[\eta_{-\nu_2},\eta_{\gamma,k}]_\kgf\otimes\eta_{\zeta,k} \\ & \hspace{5em} +(\eta_{\beta,k})_\kgf\wedge(\eta_{\gamma,k})_\kgf\otimes[\eta_{-\nu_2},\eta_{\zeta,k}]\Big),\end{align*} so the linear independence of the $\eta_{\beta,k}\otimes\eta_{\zeta,k}$ and $\eta_{\gamma,k}\otimes\eta_{\zeta,k}$ then tells us that, 
for each $k$, either $[\eta_{-\nu_2},\eta_{\beta,k}]=0$ for every $\eta_{-\nu_2}\in\mathfrak{g}_{-\nu_2}$, which contradicts $\beta-\nu_2=\nu_1\in\Delta$, or $[\eta_{-\nu_2},\eta_{\beta,k}]\in\mathfrak{g}_\gamma$, which contradicts $\beta-\gamma\not\in\Delta^+$.\qed\end{proof}

\subsection{Essential automorphisms}
We now move to the definitions of concepts specific to the idea of essential automorphisms of parabolic geometries.

For a parabolic model geometry $(G,P)$ with compatible restricted root system $(\theta,\mathfrak{c},\mathfrak{a},\Delta,\Delta^+)$, consider a homomorphism $\lambda:G_0\to\mathbb{R}_+$. Thinking of the induced Lie algebra homomorphism $\lambda_*:\mathfrak{g}_0\to\mathbb{R}$ as an element of the dual space of $\mathfrak{g}$ that vanishes on elements that are $\kgf$-orthogonal to $\mathfrak{g}_0$, we may consider the element $\lambda_*^\kgf\in\mathfrak{z}(\mathfrak{g}_0)$ uniquely determined by $\kgf(\lambda_*^\kgf,X)=\lambda_*(X)$ for $X\in\mathfrak{g}$. We use the following definitions, along the lines of \cite{CapSlovak2}.

\begin{definition}Given a homomorphism $\lambda:G_0\to\mathbb{R}_+$, we say that $\lambda_*^\kgf$ is a \emph{scaling element} if and only if $\ad_{\lambda_*^\kgf}$ restricts to multiplication by a nonzero real scalar on each $G_0$-irreducible component of $\mathfrak{p}_+$.

If $\lambda_*^\kgf$ is a scaling element and $(\mathscr{G},\omega)$ is a Cartan geometry of type $(G,P)$ over $M$, then we say that the principal $\mathbb{R}_+$-bundle \[\mathscr{G}/P_+\times_\lambda\mathbb{R}_+\cong\mathscr{G}/\ker(\lambda)P_+\] is the \emph{bundle of scales} determined by $\lambda$ for $(\mathscr{G},\omega)$.\end{definition}

For a Cartan geometry $(\mathscr{G},\omega)$ of type $(G,P)$, a \emph{Weyl structure} is just a choice of $G_0$-equivariant section $\sigma:\mathscr{G}/P_+\to\mathscr{G}$. By using $\sigma$ to pull back the Cartan connection on $\mathscr{G}$ to a $\mathfrak{g}$-valued one-form on $\mathscr{G}/P_+$ and then $\kgf_\theta$-orthogonally projecting from $\mathfrak{g}$ to $\mathfrak{g}_0$, we get a principal $G_0$-connection $\sigma^*\omega_0$ on $\mathscr{G}/P_+$, and this induces a principal $\mathbb{R}_+$-connection $\lambda_*(\sigma^*\omega_0)$ on the bundle of scales $\mathscr{G}/\ker(\lambda)P_+$ determined by $\lambda$.

We say that a Weyl structure $\sigma:\mathscr{G}/P_+\to\mathscr{G}$ for a Cartan geometry $(\mathscr{G},\omega)$ of type $(G,P)$ over $M$ is \emph{$\lambda$-exact} if and only if there exists a global section $f_\sigma:M\to\mathscr{G}/\ker(\lambda)P_+$ inducing the same principal $\mathbb{R}_+$-connection as $\sigma^*\omega_0$. By \cite{CapSlovak2}, there is a bijective correspondence between $\lambda$-exact Weyl structures $\sigma$ and global sections $f_\sigma$.

From $P$-equivariance, we get an induced action of $\Aut(\mathscr{G},\omega)$ on both the base manifold $M\cong\mathscr{G}/P$ and each bundle of scales $\mathscr{G}/\ker(\lambda)P_+$. In particular, we get an induced action of $\Aut(\mathscr{G},\omega)$ on the space of global sections of $\mathscr{G}/\ker(\lambda)P_+$ for each $\lambda$, which we can use to define essential automorphisms, along the lines of \cite{Alt}.

\begin{definition}\label{ess} Suppose $(\mathscr{G},\omega)$ is a Cartan geometry of type $(G,P)$ over $M$ and $\lambda:G_0\to\mathbb{R}_+$ is a homomorphism with $\lambda_*^\kgf$ a scaling element. An automorphism $\varphi\in\Aut(\mathscr{G},\omega)$ is \emph{$\lambda$-inessential} if and only if there exists a $\lambda$-exact Weyl structure $\sigma$ with corresponding global section $f_\sigma$ such that $\varphi\cdot f_\sigma=f_\sigma$, where $\varphi\cdot f_\sigma$ is the induced action of $\varphi$ on $f_\sigma$, and $\varphi$ is \emph{$\lambda$-essential} if and only if it is not $\lambda$-inessential. Moreover, we say that $\varphi\in\Aut(\mathscr{G},\omega)$ is \emph{essential} if and only if it is $\lambda$-essential for every $\lambda$.\end{definition}

For the purposes of this paper, we will barely need this definition. Instead, we will mainly need the following result, which is a direct consequence of Corollary 6.5 and Definition 7.11 in \cite{HolonomyPaper}, though \textit{au fond}, it is more or less already a corollary of Theorem 1.2 in \cite{Alt} or Proposition 7.12 in \cite{HolonomyPaper}.

\begin{proposition}\label{esscheck} Suppose $(\mathscr{G},\omega)$ is a Cartan geometry of type $(G,P)$ over $M$, $\lambda:G_0\to\mathbb{R}_+$ is a homomorphism such that $\lambda_*^\kgf$ is a scaling element, and $\varphi\in\Aut(\mathscr{G},\omega)$. If there exist $\mathscr{e}\in\mathscr{G}$ and $p\in P$ such that $\varphi(\mathscr{e})=\mathscr{e}p$, then $\varphi$ is $\lambda$-essential if $p$ is not conjugate to an element of $\ker(\lambda)$.\end{proposition}

In particular, if $\varphi(\mathscr{e})=\mathscr{e}a$ for some $a\in G_0<P$, then $\varphi$ is $\lambda$-essential if $a\not\in\ker(\lambda)$, since $hah^{-1}\in\ker(\lambda)$ if and only if $a\in\ker(\lambda)$ for $h\in G_0$ because $\ker(\lambda)$ is a normal subgroup of $G_0$ and $pap^{-1}=(pap^{-1}a^{-1})a\in P_+a$ for $p\in P_+$.

\section{Curvature trees}\label{forest}
Consider a parabolic model geometry $(G,P)$ with a compatible restricted root system $(\theta,\mathfrak{c},\mathfrak{a},\Delta,\Delta^+)$. Given an element $\Omega\in\Lambda^2(\mathfrak{g}/\mathfrak{p})^\vee\otimes\mathfrak{g}$, let \[\mathrm{Stab}_{G_0}(\Omega):=\{h\in G_0:h\cdot\Omega=\Omega\},\] where $p\cdot X_\kgf\wedge Y_\kgf\otimes Z=(\Ad_pX)_\kgf\wedge(\Ad_pY)_\kgf\otimes\Ad_pZ$ for $X,Y\in\mathfrak{p}_+$, $Z\in\mathfrak{g}$, and $p\in P$, and let $\mathfrak{k}_\Omega$ denote the Lie algebra of $\mathrm{Stab}_{G_0}(\Omega)$. We will use the notation $\mathrm{im}(\Omega):=\Omega(\Lambda^2\mathfrak{g}_-)\subseteq\mathfrak{g}$, $\ker(\Omega):=\{\psi\in\Lambda^2(\mathfrak{g}/\mathfrak{p}):\Omega(\psi)=0\}\subseteq\Lambda^2(\mathfrak{g}/\mathfrak{p})$, and \[\Omega\wedge\mathrm{id}:\Lambda^3(\mathfrak{g}/\mathfrak{p})\to\Lambda^2(\mathfrak{g}/\mathfrak{p})\] for the linear map given by \[X\wedge Y\wedge Z\mapsto\Omega(X\wedge Y)\wedge Z-\Omega(X\wedge Z)\wedge Y+\Omega(Y\wedge Z)\wedge X.\]

\begin{definition}We say that $\Omega\in\ker(\square)_+$ satisfies the \emph{Kruglikov-The property} if and only if $\mathrm{im}(\Omega)\subseteq\mathfrak{g}_-+\mathfrak{k}_\Omega$ and $\mathrm{im}(\Omega\wedge\mathrm{id})\subseteq\ker(\Omega)$.\end{definition}

\begin{theorem}Suppose $\Omega\in\ker(\square)_+$ satisfies the Kruglikov-The property. Then, the vector space $\mathfrak{j}_\Omega:=\mathfrak{g}_-+\mathfrak{k}_\Omega$ is a Lie algebra with respect to the bracket \[[X,Y]_{\mathfrak{j}_\Omega}:=[X,Y]-\Omega(X\wedge Y).\]\end{theorem}
\begin{proof}The bracket $[\cdot,\cdot]_{\mathfrak{j}_\Omega}$ is clearly bilinear and alternating by definition, so it just remains to prove that left-bracketing is a derivation.

Let us write $``\partial\Omega":\Lambda^3\mathfrak{j}_\Omega\to\mathfrak{g}$ for the linear map given by \begin{align*}``\partial\Omega"(X\wedge Y\wedge Z) & :=[X,\Omega(Y\wedge Z)]-[Y,\Omega(X\wedge Z)]+[Z,\Omega(X\wedge Y)] \\ & \quad -\Omega([X,Y]\wedge Z)+\Omega([X,Z]\wedge Y)-\Omega([Y,Z]\wedge X)\end{align*} when $X,Y,Z\in\mathfrak{j}_\Omega$. Then, if $X,Y,Z\in\mathfrak{g}_-<\mathfrak{j}_\Omega$, then \[``\partial\Omega"(X\wedge Y\wedge Z)=\partial\Omega(X\wedge Y\wedge Z)=0,\] since $\partial\Omega=0$. If $X,Y\in\mathfrak{j}_\Omega$ and $Z\in\mathfrak{k}_\Omega$, then \begin{align*}``\partial\Omega"(X\wedge Y\wedge Z) & =[Z,\Omega(X\wedge Y)]-\Omega([Z,X]\wedge Y)-\Omega(X\wedge[Z,Y]) \\ & =(Z\cdot\Omega)(X\wedge Y)=0,\end{align*} since $\mathfrak{k}_\Omega$ is the Lie algebra of $\mathrm{Stab}_{G_0}(\Omega)$. Thus, for all $X,Y,Z\in\mathfrak{j}_\Omega$, we have $``\partial\Omega"(X\wedge Y\wedge Z)=0$, so \begin{align*}[X,[Y,Z]_{\mathfrak{j}_\Omega}]_{\mathfrak{j}_\Omega} & =[X,[Y,Z]-\Omega(Y\wedge Z)]_{\mathfrak{j}_\Omega} \\ & =[X,[Y,Z]-\Omega(Y\wedge Z)]-\Omega(X\wedge([Y,Z]-\Omega(Y\wedge Z))) \\ & =[[X,Y],Z]+[Y,[X,Z]]-\Big([X,\Omega(Y\wedge Z)] \\ & \quad -\Omega([Y,Z]\wedge X)\Big)+\Omega(X\wedge\Omega(Y\wedge Z)) \\ & =[[X,Y],Z]+[Y,[X,Z]]+\Big([Z,\Omega(X\wedge Y)]-[Y,\Omega(X\wedge Z)] \\ & \quad -\Omega([X,Y]\wedge Z+Y\wedge [X,Z])-``\partial\Omega"(X\wedge Y\wedge Z)\Big) \\ & \quad +\Omega(X\wedge\Omega(Y\wedge Z)) \\ & =[[X,Y]_{\mathfrak{j}_\Omega},Z]_{\mathfrak{j}_\Omega}+[Y,[X,Z]_{\mathfrak{j}_\Omega}]_{\mathfrak{j}_\Omega}-\Omega((\Omega\wedge\mathrm{id})(X\wedge Y\wedge Z)) \\ & \quad -``\partial\Omega"(X\wedge Y\wedge Z) \\ & =[[X,Y]_{\mathfrak{j}_\Omega},Z]_{\mathfrak{j}_\Omega}+[Y,[X,Z]_{\mathfrak{j}_\Omega}]_{\mathfrak{j}_\Omega}.\quad\qed\end{align*}\end{proof}

Both the above theorem and its proof were, in essence, already given by Kruglikov and The in Lemma 4.1.1 of \cite{GapPhenomenon} under more specific assumptions. They then go on to prove that, when $\Omega$ is a lowest weight vector of one of the $\mathfrak{g}_0$-irreducible components of $\ker(\square)_+$ and $G$ is either complex or split-real, then there exists a \textit{local} homogeneous geometry of type $(G,P)$ whose algebra of infinitesimal symmetries contains $\mathfrak{j}_\Omega$. We would like to consider cases where this construction gives a genuine (``global") homogeneous geometry of type $(G,P)$.

\begin{definition}\label{trees} Suppose $\Omega\in\ker(\square)_+$ satisfies the Kruglikov-The property. We say that $\Omega$ is a \emph{harmonic seed} if and only if there exists a model geometry $(J_\Omega,K_\Omega)$, an isomorphism $i:K_\Omega\to\mathrm{Stab}_{G_0}(\Omega)\leq P$, and an isomorphism of $K_\Omega$-representations $\psi:\mathfrak{j}_\Omega\to\mathfrak{g}_-+\mathfrak{k}_\Omega\leq\mathfrak{g}$ such that $\mathfrak{j}_\Omega$ is the Lie algebra of $J_\Omega$, $J_\Omega/K_\Omega$ is simply connected, and $\psi|_{\mathfrak{k}_\Omega}=i_*$. In such a case, we say that the homogeneous Cartan geometry $(J_\Omega\times_{K_\Omega}P,\MC{\Omega})$ of type $(G,P)$ over $J_\Omega/K_\Omega$ induced by the extension functor $(i,\psi)$ is the \emph{curvature tree} grown from $\Omega$.\end{definition}

Writing $\MC{J_\Omega}$ for the Maurer-Cartan form of $J_\Omega$ and $\MC{P}$ for the Maurer-Cartan form of $P$, the Cartan connection $\MC{\Omega}$ of the curvature tree grown from a harmonic seed $\Omega$ is given by \[(\MC{\Omega})_{(j,p)}=\Ad_{p^{-1}}\psi(\MC{J_\Omega})+\MC{P}\] at each $(j,p)\in J_\Omega\times_{K_\Omega}P$, and its curvature at $(j,p)$ is just \[X\wedge Y\mapsto (p\cdot\Omega)(\MC{\Omega}(X)\wedge\MC{\Omega}(Y)).\] In other words, under the identification of the tangent spaces of $J_\Omega\times_{K_\Omega}P$ with $\mathfrak{g}$ given by $\MC{\Omega}$, the curvature tree grown from a harmonic seed $\Omega$ has constant curvature $\Omega$ when restricted to $J_\Omega\subseteq J_\Omega\times_{K_\Omega}P$, and because $\Omega\in\ker(\square)_+$, the curvature tree grown from $\Omega$ is a regular and normal parabolic geometry.

Of course, in order for this to be useful, we would like some sort of criterion for when a Kostant-harmonic $\Omega$ of positive homogeneity satisfying the Kruglikov-The property is a harmonic seed. We are, thus, led to our first main theorem.

\begin{theorem}\label{openingnumber} Denote by $\mathfrak{b}_-$ the nilpotent subalgebra of $\mathfrak{g}$ generated by the restricted root spaces of the negative restricted roots. If $\Omega\in\ker(\square)_+$ satisfies the Kruglikov-The property and $\mathrm{im}(\Omega)\subseteq\mathfrak{b}_-$, then $\Omega$ is a harmonic seed.\end{theorem}
\begin{proof}We will construct the Lie group $J_\Omega$ as an extension of the form \[\{e\}\to F_\Omega\hookrightarrow J_\Omega\twoheadrightarrow K_\Omega/N_\Omega\to\{e\},\] where $K_\Omega=\mathrm{Stab}_{G_0}(\Omega)$, $N_\Omega$ is a closed normal subgroup of $K_\Omega$, and $F_\Omega$ is a closed normal subgroup of $J_\Omega$.

Define $\mathfrak{f}_\Omega:=\mathfrak{g}_-+\mathrm{im}(\Omega)$, thought of as a subspace of $\mathfrak{j}_\Omega$. Then, since $K_\Omega$ stabilizes $\Omega$, and $\Omega(X\wedge Y)\in\mathrm{im}(\Omega)$ by definition, \[[\mathfrak{g}_-+\mathfrak{k}_\Omega,\mathfrak{g}_-]_{\mathfrak{j}_\Omega}\subseteq\mathfrak{g}_-+\mathrm{im}(\Omega)=\mathfrak{f}_\Omega\] and \[[\mathfrak{g}_-+\mathfrak{k}_\Omega,\mathrm{im}(\Omega)]_{\mathfrak{j}_\Omega}\subseteq\mathfrak{g}_-+\mathrm{im}(\Omega)=\mathfrak{f}_\Omega,\] so $\mathfrak{f}_\Omega$ is an ideal of $\mathfrak{j}_\Omega$. We will denote by $F_\Omega$ the simply connected Lie group with Lie algebra $\mathfrak{f}_\Omega$.

Similarly, define $\mathfrak{n}_\Omega:=(\mathfrak{g}_-+\mathrm{im}(\Omega))\cap\mathfrak{k}_\Omega$. Because $\mathfrak{g}_-+\mathrm{im}(\Omega)$ and $\mathfrak{k}_\Omega$ are preserved by the adjoint action of $K_\Omega$, $\mathfrak{n}_\Omega$ is an ideal of $\mathfrak{k}_\Omega$, and denoting by $N_\Omega$ the connected subgroup of $K_\Omega$ generated by $\mathfrak{n}_\Omega$, it also follows that $N_\Omega$ is a normal subgroup of $K_\Omega$. Moreover, because $\mathrm{im}(\Omega)\subseteq\mathfrak{b}_-$, we also have that $\mathfrak{n}_\Omega<\mathfrak{b}_-$, so every compact subgroup of $K_\Omega$ intersects $N_\Omega$ trivially, hence $N_\Omega$ is a simply connected closed normal subgroup of $K_\Omega$ and there exists a smooth (but not necessarily homomorphic) section $\sigma:K_\Omega/N_\Omega\to K_\Omega$ for the natural quotient map $q_{{}_{N_\Omega}}:K_\Omega\to K_\Omega/N_\Omega$, which we may assume maps the identity element $\bar{e}\in K_\Omega/N_\Omega$ to the identity element $e\in K_\Omega$ by left-multiplying by $\sigma(\bar{e})^{-1}\in N_\Omega$ if necessary. Note that for $\bar{k}_1,\bar{k}_2\in K_\Omega/N_\Omega$, we always have $\sigma(\bar{k}_1)\sigma(\bar{k}_2)\sigma(\bar{k}_1\bar{k}_2)^{-1}\in N_\Omega$.

Since $\mathfrak{n}_\Omega\subseteq\mathfrak{g}_-+\mathrm{im}(\Omega)=\mathfrak{f}_\Omega$, we may also consider $\mathfrak{n}_\Omega$ as a subalgebra of $\mathfrak{f}_\Omega$. In particular, $\mathfrak{n}_\Omega$ generates a connected subgroup $\overline{N_\Omega}$ of $F_\Omega$, and the adjoint representation $\Ad^\Omega$ of $F_\Omega$ on $\mathfrak{f}_\Omega$ restricts to a representation of $\overline{N_\Omega}$ on $\mathfrak{f}_\Omega$. At the level of Lie algebras, this restriction coincides with the usual adjoint action of $\mathfrak{n}_\Omega\leq\mathfrak{k}_\Omega$ on $\mathfrak{g}_-+\mathrm{im}(\Omega)$ because $\Omega(Z\wedge X)=0$ whenever $Z\in\mathfrak{k}_\Omega$, so denoting by $\alpha:N_\Omega\to\overline{N_\Omega}$ the covering homomorphism from $N_\Omega$ to $\overline{N_\Omega}$, we have that $\Ad_n|_{\mathfrak{g}_-+\mathrm{im}(\Omega)}=\Ad^\Omega_{\alpha(n)}$ for every $n\in N_\Omega$. But because $\mathfrak{n}_\Omega<\mathfrak{b}_-$, $\Ad_n$ only acts trivially on $\mathfrak{g}_-$ when $n$ is the identity element, so $\Ad^\Omega_{\alpha(n)}$ can only be the identity transformation when $n$ is the identity element. Thus, $\alpha$ is an isomorphism of Lie groups. Moreover, since the adjoint action of $\overline{N_\Omega}$ on $\mathfrak{f}_\Omega$ is faithful and unipotent, $\overline{N_\Omega}$ also cannot intersect any compact subgroups of $F_\Omega$ nontrivially, so $\overline{N_\Omega}$ is a closed subgroup of $F_\Omega$.

Consider the smooth manifold $F_\Omega\times K_\Omega/N_\Omega$. Denoting by $\cdot$ the action of $K_\Omega$ on $F_\Omega$ and by juxtaposition the group operations in $F_\Omega$ and $K_\Omega/N_\Omega$, define a map $\tilde{\alpha}:(K_\Omega/N_\Omega)^2\to\overline{N_\Omega}$ by \[\tilde{\alpha}(\bar{k}_1,\bar{k}_2):=\alpha\Big(\sigma(\bar{k}_1)\sigma(\bar{k}_2)\sigma(\bar{k}_1\bar{k}_2)^{-1}\Big)\] and a map $\mu_\Omega:(F_\Omega\times K_\Omega/N_\Omega)^2\to F_\Omega\times K_\Omega/N_\Omega$ by \[\mu_\Omega((f_1,\bar{k}_1),(f_2,\bar{k}_2)):=\Big(f_1\,(\sigma(\bar{k}_1)\cdot f_2)\,\tilde{\alpha}(\bar{k}_1,\bar{k}_2),\bar{k}_1\bar{k}_2\Big).\] We want to show that $\mu_\Omega$ gives a Lie group structure on $F_\Omega\times K_\Omega/N_\Omega$. Because $\alpha$ and $\sigma$ are smooth, $K_\Omega$ acts smoothly on $F_\Omega$, and $F_\Omega$ and $K_\Omega/N_\Omega$ are Lie groups, $\mu_\Omega$ is smooth as well, so we just have to show that $\mu_\Omega$ is a group operation. To do this, we will follow the ideas of Remark 11.1.22 of \cite{HilgertNeeb}.

Denoting by $e$ the identity element of $F_\Omega$, \[\mu_\Omega((e,\bar{e}),(f,\bar{k}))=\Big(e(\sigma(\bar{e})\cdot f)\alpha\big(\sigma(\bar{e})\sigma(\bar{k})\sigma(\bar{k})^{-1}\big),\bar{k}\Big)=(f,\bar{k}),\] so $(e,\bar{e})$ gives an identity element, and \[\mu_\Omega\Bigg(\Big(\big((\sigma(\bar{k}^{-1})\cdot f)\tilde{\alpha}(\bar{k}^{-1},\bar{k})\big)^{-1},\bar{k}^{-1}\Big),(f,\bar{k})\Bigg)=(e,\bar{e}),\] so \[(f,\bar{k})^{-1}:=\Big(\big((\sigma(\bar{k}^{-1})\cdot f)\tilde{\alpha}(\bar{k}^{-1},\bar{k})\big)^{-1},\bar{k}^{-1}\Big)\] gives an inverse map. From \begin{align*}\tilde{\alpha}(\bar{k}_1,\bar{k}_2)(\sigma(\bar{k}_1\bar{k}_2)\cdot f_3) & =(\alpha^{-1}(\tilde{\alpha}(\bar{k}_1,\bar{k}_2))\cdot\sigma(\bar{k}_1\bar{k}_2)\cdot f_3)\tilde{\alpha}(\bar{k}_1,\bar{k}_2) \\ & =(\sigma(\bar{k}_1)\sigma(\bar{k}_2)\sigma(\bar{k}_1\bar{k}_2)^{-1}\cdot\sigma(\bar{k}_1\bar{k}_2)\cdot f_3)\tilde{\alpha}(\bar{k}_1,\bar{k}_2) \\ & =(\sigma(\bar{k}_1)\sigma(\bar{k}_2)\cdot f_3)\tilde{\alpha}(\bar{k}_1,\bar{k}_2)\end{align*} and \begin{align*}\tilde{\alpha}(\bar{k}_1,\bar{k}_2)\tilde{\alpha}(\bar{k}_1\bar{k}_2,\bar{k}_3) & =\alpha(\sigma(\bar{k}_1)\sigma(\bar{k}_2)\sigma(\bar{k}_1\bar{k}_2)^{-1}\sigma(\bar{k}_1\bar{k}_2)\sigma(\bar{k}_3)\sigma(\bar{k}_1\bar{k}_2\bar{k}_3)^{-1}) \\ & =\alpha(\sigma(\bar{k}_1)\sigma(\bar{k}_2)\sigma(\bar{k}_3)\sigma(\bar{k}_1\bar{k}_2\bar{k}_3)^{-1}) \\ & =(\sigma(\bar{k}_1)\cdot\tilde{\alpha}(\bar{k}_2,\bar{k}_3))\tilde{\alpha}(\bar{k}_1,\bar{k}_2\bar{k}_3),\end{align*} we see that \[\tilde{\alpha}(\bar{k}_1,\bar{k}_2)(\sigma(\bar{k}_1\bar{k}_2)\cdot f_3)\tilde{\alpha}(\bar{k}_1\bar{k}_2,\bar{k}_3)=(\sigma(\bar{k}_1)\cdot\tilde{\alpha}(\bar{k}_2,\bar{k}_3))\tilde{\alpha}(\bar{k}_1,\bar{k}_2\bar{k}_3),\] so because $\mu_\Omega(\mu_\Omega((f_1,\bar{k}_1),(f_2,\bar{k}_2)),(f_3,\bar{k}_3))$ is given by \[\Bigg(\Big(f_1\,(\sigma(\bar{k}_1)\cdot f_2)\,\tilde{\alpha}(\bar{k}_1,\bar{k}_2)\Big)(\sigma(\bar{k}_1\bar{k}_2)\cdot f_3)\tilde{\alpha}(\bar{k}_1\bar{k}_2,\bar{k}_3),\bar{k}_1\bar{k}_2\bar{k}_3\Bigg)\] and $\mu_\Omega((f_1,\bar{k}_1),\mu_\Omega((f_2,\bar{k}_2),(f_3,\bar{k}_3)))$ is given by \[\Bigg(f_1\Big(\sigma(\bar{k}_1)\cdot\Big(f_2(\sigma(\bar{k}_2)\cdot f_3)\tilde{\alpha}(\bar{k}_2,\bar{k}_3)\Big)\Big)\tilde{\alpha}(\bar{k}_1,\bar{k}_2\bar{k}_3),\bar{k}_1\bar{k}_2\bar{k}_3\Bigg),\] this shows that $\mu_\Omega$ is associative.

Define $J_\Omega$ to be the Lie group with underlying manifold $F_\Omega\times K_\Omega/N_\Omega$ and with group operation $\mu_\Omega$. Then, $K_\Omega$ embeds as a closed subgroup of $J_\Omega$ by identifying $k\in K_\Omega$ with $(\alpha(k\sigma(q_{{}_{N_\Omega}}(k))^{-1}),q_{{}_{N_\Omega}}(k))\in J_\Omega$. Similarly, we can embed $F_\Omega$ as a closed normal subgroup of $J_\Omega$ by identifying $f\in F_\Omega$ with $(f,\bar{e})\in J_\Omega$. As subgroups of $J_\Omega$, we see that $F_\Omega\cap K_\Omega=N_\Omega$, and clearly $F_\Omega K_\Omega=J_\Omega$, so if $\tilde{\mathfrak{j}}_\Omega$ is the Lie algebra of $J_\Omega$, then we get an isomorphism of vector spaces $\psi:\tilde{\mathfrak{j}}_\Omega\to\mathfrak{g}_-+\mathfrak{k}_\Omega$. Moreover, the action of $K_\Omega$ on $F_\Omega$ corresponds to conjugation in $J_\Omega$ by construction, so $\psi$ is an isomorphism of $K_\Omega$-representations, and since the bracket of $\tilde{\mathfrak{j}}_\Omega$ coincides with the bracket of $\mathfrak{j}_\Omega$ on $\mathfrak{f}_\Omega\leq\tilde{\mathfrak{j}}_\Omega$ and on $\mathfrak{k}_\Omega\leq\tilde{\mathfrak{j}}_\Omega$, $\tilde{\mathfrak{j}}_\Omega$ is isomorphic to $\mathfrak{j}_\Omega$.

Letting $i:(\alpha(k\sigma(q_{{}_{N_\Omega}}(k))^{-1}),q_{{}_{N_\Omega}}(k))\mapsto k$, we have that $\psi|_{\mathfrak{k}_\Omega}=i_*$, and because $F_\Omega$ is simply connected and $N_\Omega$ is connected, $J_\Omega/K_\Omega\cong F_\Omega/N_\Omega$ is simply connected. Thus, $(J_\Omega,K_\Omega)$ and $(i,\psi)$ satisfy the conditions in Definition \ref{trees}, so $\Omega$ is a harmonic seed.\qed\end{proof}

In the rest of the paper, we will be primarily focused on lowest weight vectors of the form $\Omega=\sum_k(\eta_{\beta,k})_\kgf\wedge(\eta_{\gamma,k})_\kgf\otimes\eta_{\zeta,k}$ in $\ker(\square)_+$, for which we can considerably simplify and strengthen the above theorem as follows.

\begin{theorem}\label{main} Suppose $\Omega=\sum_k(\eta_{\beta,k})_\kgf\wedge(\eta_{\gamma,k})_\kgf\otimes\eta_{\zeta,k}\in\ker(\square)_+$ is a lowest weight vector, where $\beta,\gamma\in\Delta^+(\mathfrak{p}_+)$ and $\zeta\in\Delta\setminus\Delta^+(\mathfrak{p}_+)$ with $\eta_{\beta,k}\in\mathfrak{g}_\beta$, $\eta_{\gamma,k}\in\mathfrak{g}_\gamma$, and $\eta_{\zeta,k}\in\mathfrak{g}_\zeta$ for each $k$. If $\zeta\not\in\{-\beta,-\gamma\}$, then $\Omega$ is a harmonic seed with $J_\Omega/K_\Omega$ diffeomorphic to $\mathbb{R}^{\dim(\mathfrak{g}_-)}$.\end{theorem}
\begin{proof}From Lemma \ref{assume}, $\zeta\in\Delta\setminus\Delta^+(\mathfrak{p}_+)$ implies $\zeta\in-\Delta^+$ and, by swapping $\beta$ and $\gamma$ if necessary, we may assume $\beta$ is a simple restricted root. Since $\zeta\not\in\{-\beta,-\gamma\}$, $\mathrm{im}(\Omega\wedge\mathrm{id})\subseteq\ker(\Omega)$, and since $\zeta\in -\Delta^+$ and $\Omega$ is a lowest weight vector, either $\mathfrak{g}_\zeta\subseteq\mathfrak{g}_-$ or $\mathfrak{g}_\zeta\subseteq\mathfrak{k}_\Omega$. Thus, $\Omega$ satisfies the Kruglikov-The property, and because $\zeta\in -\Delta^+$, we have $\mathrm{im}(\Omega)\subseteq\mathfrak{g}_\zeta\subseteq\mathfrak{b}_-$, so $\Omega$ is a harmonic seed by Theorem \ref{openingnumber}.

It just remains to show that $J_\Omega/K_\Omega$ is diffeomorphic to $\mathbb{R}^{\dim(\mathfrak{g}_-)}$. To do this, we will show that the ideal $\mathfrak{f}_\Omega\leq\mathfrak{j}_\Omega$ from the proof of Theorem \ref{openingnumber} is solvable; then, $J_\Omega/K_\Omega\cong F_\Omega/N_\Omega$ is the quotient of a simply connected solvable Lie group by a connected closed subgroup, hence it must be diffeomorphic to $\mathbb{R}^{\dim(F_\Omega/N_\Omega)}=\mathbb{R}^{\dim(\mathfrak{g}_-)}$.

Let us recursively define $D_\Omega^i(\mathfrak{f}_\Omega)$ by $D_\Omega^0(\mathfrak{f}_\Omega):=\mathfrak{f}_\Omega$ and \[D_\Omega^{i+1}(\mathfrak{f}_\Omega):=[D_\Omega^i(\mathfrak{f}_\Omega),D_\Omega^i(\mathfrak{f}_\Omega)]_{\mathfrak{j}_\Omega}.\] To show that $D_\Omega^i(\mathfrak{j}_\Omega)=\{0\}$ for large enough $i$, which is what we want to prove, it is sufficient to show that $\kgf(\eta_{\beta,k},Y)=0$ for each $k$ and every $Y\in D_\Omega^1(\mathfrak{f}_\Omega)$, since $D_\Omega^i(\mathfrak{f}_\Omega)\subseteq D_\Omega^1(\mathfrak{f}_\Omega)$ for every $i>0$ and hence \[D_\Omega^{i+1}(\mathfrak{f}_\Omega)=[D_\Omega^i(\mathfrak{f}_\Omega),D_\Omega^i(\mathfrak{f}_\Omega)]_{\mathfrak{j}_\Omega}=[D_\Omega^i(\mathfrak{f}_\Omega),D_\Omega^i(\mathfrak{f}_\Omega)],\] which then must be $\{0\}$ for sufficiently large $i$ because $\mathfrak{g}_-+\mathrm{im}(\Omega)$ is nilpotent as a subalgebra of $\mathfrak{g}$.

Suppose $Y\in D_\Omega^1(\mathfrak{f}_\Omega)$. Then, since $D_\Omega^1(\mathfrak{f}_\Omega)\subseteq [\mathfrak{f}_\Omega,\mathfrak{f}_\Omega]+\mathrm{im}(\Omega)$, we can write $Y$ in the form $Y=Y'+Z$ for some $Y'\in[\mathfrak{g}_-+\mathrm{im}(\Omega),\mathfrak{g}_-+\mathrm{im}(\Omega)]$ and $Z\in\mathrm{im}(\Omega)$. But $\beta$ is simple, so since $\mathfrak{g}_-+\mathrm{im}(\Omega)\subseteq\mathfrak{b}_-$, $\kgf(\eta_{\beta,k},Y')=0$, and since $\zeta\not\in\{-\beta,-\gamma\}$, $\kgf(\eta_{\beta,k},Z)=0$ as well. Thus, $\kgf(\eta_{\beta,k},Y)=0$, so $\mathfrak{f}_\Omega$ is solvable.\qed\end{proof}

In general, we suspect that $\zeta\not\in\{-\beta,-\gamma\}$ always holds for lowest weight vectors of the form $\Omega=\sum_k(\eta_{\beta,k})_\kgf\wedge(\eta_{\gamma,k})_\kgf\otimes\eta_{\zeta,k}$. Indeed, when $G$ is split-real and simple of rank greater than 2, this is the case.

\begin{corollary}\label{seedcheck} If $G$ is split-real and simple of rank greater than $2$, then every lowest weight vector of the form $\Omega=\sum_k(\eta_{\beta,k})_\kgf\wedge(\eta_{\gamma,k})_\kgf\otimes\eta_{\zeta,k}\in\ker(\square)_+$ is a harmonic seed with $J_\Omega/K_\Omega$ diffeomorphic to $\mathbb{R}^{\dim(\mathfrak{g}_-)}$.\end{corollary}
\begin{proof}Because $G$ is split-real, all of the root spaces are 1-dimensional, so if $\Omega=\sum_k(\eta_{\beta,k})_\kgf\wedge(\eta_{\gamma,k})_\kgf\otimes\eta_{\zeta,k}$, then $\Omega$ must actually be of the form \[\Omega=(\eta_\beta)_\kgf\wedge(\eta_\gamma)_\kgf\otimes\eta_\zeta.\] Moreover, for such lowest weight vectors, it is shown in Lemma 4.1.2 of \cite{GapPhenomenon} that $\zeta\not\in -\Delta^+$ only happens for rank $2$ and $1$, so $\zeta\in -\Delta^+$.

Because each root space is $1$-dimensional and \begin{align*}0 & =\partial^*((\eta_\beta)_\kgf\wedge(\eta_\gamma)_\kgf\otimes\eta_\zeta) \\ & =(\eta_\beta)_\kgf\otimes[\eta_\gamma,\eta_\zeta]-(\eta_\gamma)_\kgf\otimes[\eta_\beta,\eta_\zeta]-[\eta_\beta,\eta_\gamma]_\kgf\otimes\eta_\zeta,\end{align*} $\beta+\gamma$, $\beta+\zeta$, and $\gamma+\zeta$ cannot be roots, nor can they be $0$. In particular, $\zeta\not\in\{-\beta,-\gamma\}$, so $\Omega$ is a harmonic seed with $J_\Omega/K_\Omega$ diffeomorphic to $\mathbb{R}^{\dim(\mathfrak{g}_-)}$ by Theorem \ref{main}.\qed\end{proof}

In particular, this gives us non-flat regular normal homogeneous Cartan geometries of type $(G,P)$ for every Yamaguchi-nonrigid parabolic model geometry $(G,P)$ with $G$ split-real and simple of rank greater than $2$.

\section{Essential automorphisms on non-flat geometries}\label{spotlight}
\subsection{Outline}
For a given parabolic model geometry $(G,P)$ with compatible restricted root system $(\theta,\mathfrak{c},\mathfrak{a},\Delta,\Delta^+)$, we would like to construct a regular normal Cartan geometry $(\mathscr{G},\omega)$ of type $(G,P)$ on a (closed) manifold $M$ such that $(\mathscr{G},\omega)$ is not flat but still admits an essential automorphism. For this purpose, Proposition \ref{esscheck} gives a useful way of checking whether a given automorphism $\varphi\in\Aut(\mathscr{G},\omega)$ is essential when $\varphi$ has a fixed point on the base manifold, since if $\varphi(\mathscr{e})=\mathscr{e}a$ for some $\mathscr{e}\in\mathscr{G}$ and $a\in G_0$, then all we need to do is prove that $a\not\in\ker(\lambda)$ for every $\lambda$ such that $\lambda_*^\kgf$ is a scaling element.

Therefore, to start, we will take the curvature tree $(J_\Omega\times_{K_\Omega}P,\MC{\Omega})$ grown from a harmonic seed of the form $\Omega=\sum_k(\eta_{\beta,k})_{\kgf}\wedge(\eta_{\gamma,k})_\kgf\otimes\eta_{\zeta,k}\in\ker(\square)_+$ as in Theorem \ref{main}, and apply Proposition \ref{esscheck} to show that (under mild assumptions) there is a one-parameter subgroup $a_\mathbb{R}<J_\Omega$ of essential automorphisms fixing a point $o\in J_\Omega/K_\Omega$, which we might think of as ``the origin". However, $J_\Omega/K_\Omega\cong G_-\cong\mathbb{R}^{\dim(\mathfrak{g}_-)}$ is not compact, so we still need to modify this somehow to get a closed base manifold.

To this end, suppose $a_\mathbb{R}$ also fixes a point $u_0\in(J_\Omega/K_\Omega)\setminus\{o\}$. We will consider another one-parameter subgroup $c_\mathbb{R}<J_\Omega$ that both commutes with $a_\mathbb{R}$ and shrinks every point of $J_\Omega/K_\Omega$ to $o$, meaning that $c_t(u)\rightarrow o$ as $t\rightarrow+\infty$ for every $u\in J_\Omega/K_\Omega$. Removing the point $o\in J_\Omega/K_\Omega$ and then quotienting by the subgroup $c_{t_0\mathbb{Z}}$ for $t_0\neq 0$ gives a non-flat regular normal Cartan geometry of type $(G,P)$ over $c_{t_0\mathbb{Z}}\backslash((J_\Omega/K_\Omega)\setminus\{o\})\cong S^1\times S^{\dim(\mathfrak{g}_-)-1}$, which is closed, and $a_\mathbb{R}$ descends to a one-parameter subgroup of automorphisms of this new geometry because it commutes with $c_\mathbb{R}$, so we can apply Proposition \ref{esscheck} using a point over the image of $u_0$ in the quotient to see that $a_\mathbb{R}$ is still essential.

\begin{figure}
\centering\includegraphics[width=0.4\textwidth]{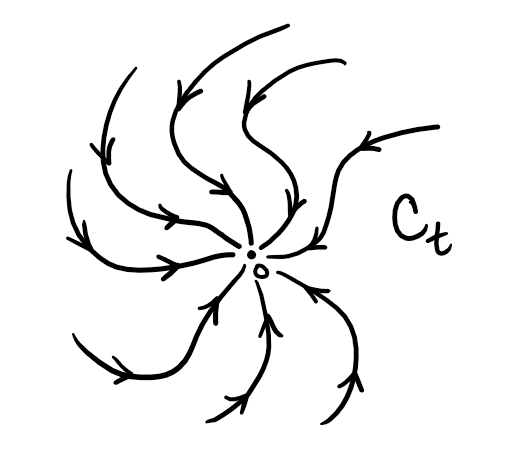}
\includegraphics[width=0.4\textwidth]{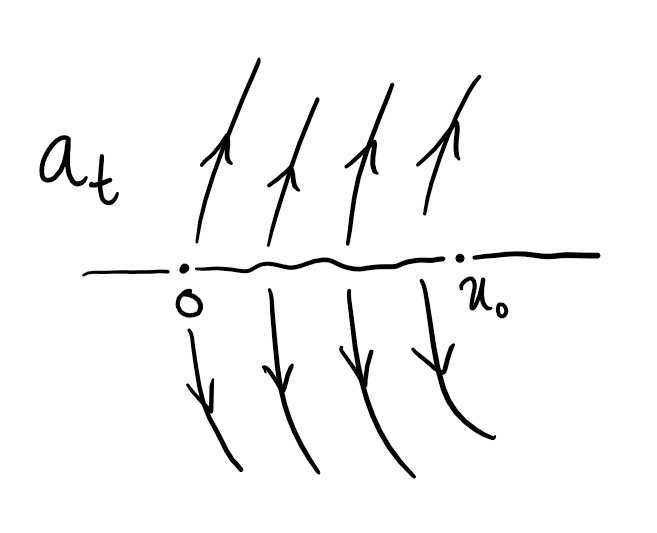}
\caption{$c_t$ shrinks everything toward $o$, while $a_t$ is essential and fixes both $o$ and $u_0$}
\label{flowpicture}
\end{figure}

Finally, we will show that the geometries over $c_{t_0\mathbb{Z}}\backslash((J_\Omega/K_\Omega)\setminus\{o\})$ are not isomorphic for distinct choices of $t_0>0$, so this gives an infinite family of distinct non-flat Cartan geometries over a closed base manifold with essential automorphisms.

\subsection{Essential automorphisms}
Let us fix a lowest weight vector $\Omega=\sum_k(\eta_{\beta,k})_{\kgf}\wedge(\eta_{\gamma,k})_\kgf\otimes\eta_{\zeta,k}\in\ker(\square)_+$, where $\beta,\gamma\in\Delta^+(\mathfrak{p}_+)$ and $\zeta\in\Delta\setminus(\Delta^+(\mathfrak{p}_+)\cup\{-\beta,-\gamma\})$ so that $\Omega$ is a harmonic seed with $J_\Omega/K_\Omega\cong\mathbb{R}^{\dim(\mathfrak{g}_-)}$ by Theorem \ref{main}, and a point $o:=eK_\Omega\in J_\Omega/K_\Omega$. The subgroup of $J_\Omega\simeq\Aut(J_\Omega\times_{K_\Omega}P,\omega)$ that fixes $o$ is precisely $K_\Omega$, which we identify with $\mathrm{Stab}_{G_0}(\Omega)<G_0<P$.

Suppose $\dot{a}_0\in\mathfrak{a}$. Then, our choice of $\Omega$ tells us that \begin{align*}\dot{a}_0\cdot\Omega & =\dot{a}_0\cdot\sum_k(\eta_{\beta,k})_{\kgf}\wedge(\eta_{\gamma,k})_\kgf\otimes\eta_{\zeta,k} \\ & =\sum_k\Big((\ad_{\dot{a}_0}\eta_{\beta,k})_{\kgf}\wedge(\eta_{\gamma,k})_\kgf\otimes\eta_{\zeta,k} \\ & \hspace{5em} +(\eta_{\beta,k})_{\kgf}\wedge(\ad_{\dot{a}_0}\eta_{\gamma,k})_\kgf\otimes\eta_{\zeta,k} \\ & \hspace{5em} +(\eta_{\beta,k})_{\kgf}\wedge(\eta_{\gamma,k})_\kgf\otimes\ad_{\dot{a}_0}\eta_{\zeta,k}\Big) \\ & =(\beta(\dot{a}_0)+\gamma(\dot{a}_0)+\zeta(\dot{a}_0))\Omega,\end{align*} so $\mathfrak{a}\cap\mathfrak{k}_\Omega=\ker(\beta+\gamma+\zeta)$. Therefore, if we can find $\dot{a}_0\in\ker(\beta+\gamma+\zeta)$ such that $\lambda_*(\dot{a}_0)\neq 0$ for each $\lambda:G_0\to\mathbb{R}_+$ with $\lambda_*^\kgf$ a scaling element, then for $a_t=\exp(t\dot{a}_0)$, $\lambda(a_t)\neq 1$ for each $t\neq 0$ since $\lambda(\exp(t\dot{a}_0))=e^{t\lambda_*(\dot{a}_0)}$, so by Proposition \ref{esscheck}, $a_t$ is an essential automorphism.

If $Z\in\mathfrak{z}(\mathfrak{g}_0)$ is a scaling element, then by definition it acts by multiplication by a nonzero real scalar on each $G_0$-irreducible component of $\mathfrak{p}_+$, so we must have $Z\in\mathfrak{z}(\mathfrak{g}_0)\cap\mathfrak{a}$ and, for $\alpha\in\Delta^+(\mathfrak{p}_+)$, the real scalar by which it acts on the $G_0$-irreducible component containing $\mathfrak{g}_\alpha$ is $\alpha(Z)$. Thus, the set of scaling elements is precisely \[\left\{Z\in\mathfrak{z}(\mathfrak{g}_0)\cap\mathfrak{a}:\alpha(Z)\neq 0\text{ for each }\alpha\in\Delta^+(\mathfrak{p}_+)\right\},\] and if $\lambda_*^\kgf$ is a scaling element, then \[\alpha(\lambda_*^\kgf)=\kgf(\alpha^\kgf,\lambda_*^\kgf)=\lambda_*(\alpha^\kgf)\neq 0\] for each $\alpha\in\Delta^+(\mathfrak{p}_+)$. In particular, if we can find $\alpha\in\Delta^+(\mathfrak{p}_+)$ and \[R\in\mathfrak{g}_0^\text{ss}\cap\mathfrak{a}=\{R\in\mathfrak{a}:\kgf(R,Z)=0\text{ for all }Z\in\mathfrak{z}(\mathfrak{g}_0)\cap\mathfrak{a}\}\] such that $\dot{a}_0=\alpha^\kgf+R\in\ker(\beta+\gamma+\zeta)$, then each $a_t=\exp(t\dot{a}_0)$ with $t\neq 0$ is an essential automorphism, since $\lambda(\exp(t\dot{a}_0))=e^{t\lambda_*(\dot{a}_0)}$ and $\lambda_*(\dot{a}_0)=\lambda_*(\alpha^\kgf)\neq 0$. Of course, no such $\dot{a}_0$ can exist when $(\beta+\gamma+\zeta)^\kgf$ is itself a scaling element, since then $(\beta+\gamma+\zeta)(\alpha^\kgf+R)=\alpha((\beta+\gamma+\zeta)^\kgf)\neq 0$, but this is the only obstacle.

\begin{proposition}\label{scalobstr} Suppose $(\beta+\gamma+\zeta)^\kgf$ is not a scaling element. Then, there exist $\alpha\in\Delta^+(\mathfrak{p}_+)$ and $R\in\mathfrak{g}_0^\text{ss}\cap\mathfrak{a}$ such that $\alpha^\kgf+R\in\ker(\beta+\gamma+\zeta)$.\end{proposition}
\begin{proof}Suppose $(\beta+\gamma+\zeta)^\kgf$ is not a scaling element. If $(\beta+\gamma+\zeta)^\kgf\in\mathfrak{z}(\mathfrak{g}_0)\cap\mathfrak{a}$, then there is some $\alpha\in\Delta^+(\mathfrak{p}_+)$ such that $\alpha((\beta+\gamma+\zeta)^\kgf)=(\beta+\gamma+\zeta)(\alpha^\kgf)=0$, so $\alpha^\kgf+0\in\ker(\beta+\gamma+\zeta)$. If $(\beta+\gamma+\zeta)^\kgf\not\in\mathfrak{z}(\mathfrak{g}_0)\cap\mathfrak{a}$, then consider the (nonzero) $\kgf$-orthogonal projection $\mathrm{pr}_{\mathfrak{g}_0^\text{ss}\cap\mathfrak{a}}(\beta+\gamma+\zeta)^\kgf\in\mathfrak{g}_0^\text{ss}\cap\mathfrak{a}$. Then, 
for arbitrary $\alpha\in\Delta^+(\mathfrak{p}_+)$, \[R=\frac{-\kgf(\beta+\gamma+\zeta,\alpha)}{(\beta+\gamma+\zeta)\left(\mathrm{pr}_{\mathfrak{g}_0^\text{ss}\cap\mathfrak{a}}(\beta+\gamma+\zeta)^\kgf\right)}\mathrm{pr}_{\mathfrak{g}_0^\text{ss}\cap\mathfrak{a}}(\beta+\gamma+\zeta)^\kgf\] satisfies $\alpha^\kgf+R\in\ker(\beta+\gamma+\zeta)$.\qed\end{proof}


In particular, if $(\beta+\gamma+\zeta)^\kgf$ is not a scaling element, then the $\dot{a}_0=\alpha^\kgf+R$ constructed in Proposition \ref{scalobstr} exponentiates to give essential automorphisms of the curvature tree grown from $\Omega$.

\subsection{Compact quotients}
Now, we want to find an $a_\mathbb{R}\leq K_\Omega$, possibly different from the one constructed in Proposition \ref{scalobstr}, and a point $u_0\in(J_\Omega/K_\Omega)\setminus\{o\}$ such that $a_t(u_0)=u_0$, so that we can still apply Proposition \ref{esscheck} after removing $o$. To do this, suppose we have already chosen $\dot{a}_0\in\ker(\beta+\gamma+\zeta)$ and that there is some $\nu_0\in\Delta^+(\mathfrak{p}_+)$ such that $\nu_0(\dot{a}_0)=0$. Then, for each $X\in\mathfrak{g}_{-\nu_0}$, where $\mathfrak{g}_{-\nu_0}$ is thought of as a subspace of $\mathfrak{j}_\Omega$, \begin{align*}a_t\exp(X) & =a_t\exp(X)a_t^{-1}a_t=\exp(\Ad_{a_t}X)a_t \\ & =\exp(e^{-t\nu_0(\dot{a}_0)}X)a_t=\exp(X)a_t\in\exp(X)P,\end{align*} so $u_0=\exp(X)K_\Omega\in J_\Omega/K_\Omega$ is another point fixed by $a_t$. Thus, following the construction in the previous section, we can construct a one-parameter subgroup $a_\mathbb{R}$ of essential automorphisms for the curvature tree grown from $\Omega$ that fix both $o$ and a point $u_0\in (J_\Omega/K_\Omega)\setminus\{o\}$ by finding a pair of restricted roots $\alpha,\nu_0\in\Delta^+(\mathfrak{p}_+)$ and an element $R\in\mathfrak{g}_0^\text{ss}\cap\mathfrak{a}$ such that \[\alpha^\kgf+R\in\ker(\beta+\gamma+\zeta)\cap\ker(\nu_0)\] and then defining $\dot{a}_0=\alpha^\kgf+R$.

To get compact quotients, we also need a one-parameter subgroup $c_\mathbb{R}<J_\Omega$ that both commutes with $a_\mathbb{R}$ and shrinks every point of $J_\Omega/K_\Omega$ to $o$. This has a much more straightforward solution than finding an essential automorphism: find $\dot{c}_0\in\ker(\beta+\gamma+\zeta)$ such that $\nu(\dot{c}_0)>0$ for all $\nu\in\Delta^+(\mathfrak{p}_+)$. Then, for every point $\exp(X_1)\cdots\exp(X_\ell)K_\Omega\in J_\Omega/K_\Omega$, with each $X_i\in\mathfrak{g}_{-\nu_i}<\mathfrak{j}_\Omega$ for some $\nu_i\in\Delta^+(\mathfrak{p}_+)$, we get \begin{align*}c_t\exp(X_1)\cdots\exp(X_\ell) & =c_t\exp(X_1)\cdots\exp(X_\ell)c_t^{-1}c_t \\ & =\exp(\Ad_{c_t}X_1)\cdots\exp(\Ad_{c_t}X_\ell)c_t \\ & =\exp(e^{-t\nu_1(\dot{c}_0)}X_1)\cdots\exp(e^{-t\nu_\ell(\dot{c}_0)}X_\ell)c_t,\end{align*} so as $t\rightarrow+\infty$, $c_t\exp(X_1)\cdots\exp(X_\ell)K_\Omega\rightarrow o$, and $\ker(\beta+\gamma+\zeta)<\mathfrak{a}$ is an abelian subalgebra, so $c_\mathbb{R}$ commutes with $a_\mathbb{R}$. 

Condensing the above discussion into a single result, we get the following.

\begin{theorem}\label{stumps} Suppose $\Omega=\sum_k(\eta_{\beta,k})_{\kgf}\wedge(\eta_{\gamma,k})_\kgf\otimes\eta_{\zeta,k}$ is a harmonic seed as above and $\dot{a}_0,\dot{c}_0\in\ker(\beta+\gamma+\zeta)<\mathfrak{a}$ such that $\nu(\dot{c}_0)>0$ for all $\nu\in\Delta^+(\mathfrak{p}_+)$ and $\dot{a}_0=\alpha^\kgf+R\in\ker(\nu_0)$ for some $\alpha,\nu_0\in\Delta^+(\mathfrak{p}_+)$ and $R\in\mathfrak{g}_0^\text{ss}\cap\mathfrak{a}$. Then, for each $t_0>0$, $a_\mathbb{R}$ descends to a one-parameter family of essential automorphisms of the Cartan geometry of type $(G,P)$ on $c_{t_0\mathbb{Z}}\backslash((J_\Omega\times_{K_\Omega}P)\setminus\{q_{{}_P}^{-1}(o)\})$ over $c_{t_0\mathbb{Z}}\backslash((J_\Omega/K_\Omega)\setminus\{o\})\cong S^1\times S^{\dim(\mathfrak{g}_-)-1}$ induced by the curvature tree grown from $\Omega$.\end{theorem}

\subsection{Strategies for finding $\dot{a}_0$ and $\dot{c}_0$}\label{acstrat}
Unlike in the noncompact case, where Proposition \ref{scalobstr} guarantees the existence of $\dot{a}_0$ satisfying the desired conditions as long as $(\beta+\gamma+\zeta)^\kgf$ is not a scaling element, Theorem \ref{stumps} assumes the existence of $\dot{a}_0$ and $\dot{c}_0$ without guaranteeing that they exist. This is irritatingly necessary, as we were unable to prove a result directly analogous to Proposition \ref{scalobstr} for $\dot{a}_0$ and $\dot{c}_0$. We \textit{suspect} that they always exist provided $(\beta+\gamma+\zeta)^\kgf$ is not a scaling element and the real rank of $G$ is at least $3$, and in principle, we could verify that this is true by using Yamaguchi's list and checking each and every case, but this would not really help us understand why the result is true.

There are, however, some general strategies for finding $\dot{a}_0$ and $\dot{c}_0$ that, as we will see in Section \ref{examples}, still allow us to apply Theorem \ref{stumps} to many, \textit{many} examples.

Throughout this section, let us assume that $G$ is simple with real rank at least $3$. The condition on the rank is necessary: $\dot{a}_0$ and $\dot{c}_0$ are linearly independent if they exist, so $\ker(\beta+\gamma+\zeta)$ must have dimension at least $2$, and $\ker(\beta+\gamma+\zeta)$ has codimension $1$ in $\mathfrak{a}$, so $\mathfrak{a}$ must have dimension at least $3$.

To start, let us look for valid choices of $\dot{a}_0$. Just as in the noncompact case, the main obstruction comes from the scaling elements. We split this into two cases: when $(\beta+\gamma+\zeta)^\kgf\in\mathfrak{z}(\mathfrak{g}_0)$ but is not a scaling element, and when $(\beta+\gamma+\zeta)^\kgf\not\in\mathfrak{z}(\mathfrak{g}_0)$, in which case it already cannot be a scaling element because the set of scaling elements is contained in $\mathfrak{z}(\mathfrak{g}_0)$.

\begin{proposition}\label{centralnotscaling} Suppose $(\beta+\gamma+\zeta)^\kgf\in\mathfrak{z}(\mathfrak{g}_0)\cap\mathfrak{a}$ is not a scaling element. Then, there exist restricted roots $\alpha,\nu_0\in\Delta^+(\mathfrak{p}_+)$ and $R\in\mathfrak{g}_0^\text{ss}\cap\mathfrak{a}$ such that $\alpha^\kgf+R\in\ker(\beta+\gamma+\zeta)\cap\ker(\nu_0)$.\end{proposition}
\begin{proof}Because $(\beta+\gamma+\zeta)^\kgf\in\mathfrak{z}(\mathfrak{g}_0)\cap\mathfrak{a}$ is not a scaling element, there is some $\alpha\in\Delta^+(\mathfrak{p}_+)$ such that $\alpha((\beta+\gamma+\zeta)^\kgf)=(\beta+\gamma+\zeta)(\alpha^\kgf)=0$.

If $\mathfrak{a}\subseteq\mathfrak{z}(\mathfrak{g}_0)$, then $\Delta^+(\mathfrak{p}_+)=\Delta^+$, so because we assume $\mathfrak{g}$ has real rank greater than $2$, we can just pick a $\nu_0\in\Delta^+$ that is $\kgf$-orthogonal to $\alpha$ to get $\alpha^\kgf+0\in\ker(\beta+\gamma+\zeta)\cap\ker(\nu_0)$. 
If $\mathfrak{a}\not\subseteq\mathfrak{z}(\mathfrak{g}_0)$, then for any $\nu_0\in\Delta^+(\mathfrak{p}_+)$ with $\nu_0^\kgf\not\in\mathfrak{z}(\mathfrak{g}_0)$, we can choose \[R=\frac{-\kgf(\nu_0,\alpha)}{\nu_0\left(\mathrm{pr}_{\mathfrak{g}_0^\text{ss}\cap\mathfrak{a}}(\nu_0^\kgf)\right)}\mathrm{pr}_{\mathfrak{g}_0^\text{ss}\cap\mathfrak{a}}(\nu_0^\kgf).\quad\qed\]\end{proof}

For $(\beta+\gamma+\zeta)^\kgf\not\in\mathfrak{z}(\mathfrak{g}_0)$, things get slightly more difficult. Thankfully, for $\dim(\mathfrak{g}_0^\text{ss}\cap\mathfrak{a})>1$, we can still just use elementary linear algebra.

\begin{proposition}\label{bigg0case} Suppose $(\beta+\gamma+\zeta)^\kgf\not\in\mathfrak{z}(\mathfrak{g}_0)$ and $\dim(\mathfrak{g}_0^\text{ss}\cap\mathfrak{a})>1$. Then, for each $\alpha\in\Delta^+(\mathfrak{p}_+)$, there exist $R\in\mathfrak{g}_0^\text{ss}\cap\mathfrak{a}$ and $\nu_0\in\Delta^+(\mathfrak{p}_+)$ such that $\alpha^\kgf+R\in\ker(\beta+\gamma+\zeta)\cap\ker(\nu_0)$.\end{proposition}
\begin{proof}Since $\dim(\mathfrak{g}_0^\text{ss}\cap\mathfrak{a})>1$, choose $\nu_0\in\Delta^+(\mathfrak{p}_+)$ such that $\mathrm{pr}_{\mathfrak{g}_0^\text{ss}\cap\mathfrak{a}}(\nu_0^\kgf)$ is not in the span of $\mathrm{pr}_{\mathfrak{g}_0^\text{ss}\cap\mathfrak{a}}(\beta+\gamma+\zeta)^\kgf$. 
Then, there exists at least one $R\in\mathfrak{g}_0^\text{ss}\cap\mathfrak{a}$ solving the system of linear equations given by $(\beta+\gamma+\zeta)(R)=-\kgf(\beta+\gamma+\zeta,\alpha)$ and $\nu_0(R)=-\kgf(\nu_0,\alpha)$, and any such $R$ satisfies $\alpha^\kgf+R\in\ker(\beta+\gamma+\zeta)\cap\ker(\nu_0)$.\qed\end{proof}

If $\dim(\mathfrak{g}_0^\text{ss}\cap\mathfrak{a})=1$, then we no longer have the dimensional elbow room to get $R$ from a system of two linear equations. However, in this case $(\mathfrak{g}_0^\text{ss}\cap\mathfrak{a})_\kgf$ is spanned by one of the roots in $\Delta(\mathfrak{g}_0)$, and if we assume the real rank is greater than $2$, then there will always be at least one restricted root $\kgf$-orthogonal to one of---and hence all of---the restricted roots in $\Delta(\mathfrak{g}_0)$. If we assume that $\alpha$ and $\nu_0$ are a $\kgf$-orthogonal pair of such restricted roots, then for \[R=-\frac{-\kgf(\beta+\gamma+\zeta,\alpha)}{(\beta+\gamma+\zeta)(\mathrm{pr}_{\mathfrak{g}_0^\text{ss}\cap\mathfrak{a}}(\beta+\gamma+\zeta)^\kgf)}\mathrm{pr}_{\mathfrak{g}_0^\text{ss}\cap\mathfrak{a}}(\beta+\gamma+\zeta)^\kgf,\] we have $\alpha^\kgf+R\in\ker(\beta+\gamma+\zeta)\cap\ker(\nu_0)$, and for sufficiently high real rank, such a pair always exists
.

Thus, for sufficiently high real rank, we get a one-parameter family of essential automorphisms $a_\mathbb{R}$ that fix both $o$ and a point $u_0\neq o$ whenever $(\beta+\gamma+\zeta)^\kgf$ is not a scaling element.

Now, we move on to strategies for finding $\dot{c}_0$. For $\nu\in\Delta^+(\mathfrak{p}_+)$, $\nu(E_\text{gr})$ is always going to be a positive integer, so an obvious strategy is to try to produce $\dot{c}_0$ from $E_\text{gr}$. Specifically, we want to make $\dot{c}_0=E_\text{gr}-S$ for some $S\in\mathfrak{a}$ such that $(\beta+\gamma+\zeta)(S)=(\beta+\gamma+\zeta)(E_\text{gr})$ and $\nu(S)<\nu(E_\text{gr})$ for each $\nu\in\Delta^+(\mathfrak{p}_+)$. Since $\mathfrak{p}_+$ is generated by the grading component $\mathfrak{g}_1$, this latter condition is equivalent to requiring that $\nu(S)<1$ for each $\nu\in\Delta^+(\mathfrak{p}_+)$ with $\nu(E_\text{gr})=1$.

A particularly convenient choice of $S$ would be the $\kgf$-orthogonal projection of $E_\text{gr}$ onto the span of $(\beta+\gamma+\zeta)^\kgf$, so that $E_\text{gr}-S$ would be the $\kgf$-orthogonal projection of $E_\text{gr}$ onto $\ker(\beta+\gamma+\zeta)$. The following proposition tells us when this choice is valid.

\begin{proposition}\label{shrink} Let $\dot{c}_0=\mathrm{pr}_{\ker(\beta+\gamma+\zeta)}(E_\text{gr})$ be the $\kgf$-orthogonal projection of the grading element onto $\ker(\beta+\gamma+\zeta)$. Then, $\nu(\dot{c}_0)>0$ for all $\nu\in\Delta^+(\mathfrak{p}_+)$ if and only if \[\kgf(\beta+\gamma+\zeta,\nu)<\frac{\kgf(\beta+\gamma+\zeta,\beta+\gamma+\zeta)}{\beta(E_\text{gr})+\gamma(E_\text{gr})+\zeta(E_\text{gr})}\] for all $\nu$ with $\nu(E_\text{gr})=1$.\end{proposition}
\begin{proof}Since \[\dot{c}_0=E_\text{gr}-\frac{\beta(E_\text{gr})+\gamma(E_\text{gr})+\zeta(E_\text{gr})}{\kgf(\beta+\gamma+\zeta,\beta+\gamma+\zeta)}(\beta+\gamma+\zeta)^\kgf,\] for $\nu(E_\text{gr})=1$ we get \[\nu(\dot{c}_0)=1-\frac{\beta(E_\text{gr})+\gamma(E_\text{gr})+\zeta(E_\text{gr})}{\kgf(\beta+\gamma+\zeta,\beta+\gamma+\zeta)}\kgf(\nu,\beta+\gamma+\zeta),\] which is greater than $0$ if and only if \[\kgf(\beta+\gamma+\zeta,\nu)<\frac{\kgf(\beta+\gamma+\zeta,\beta+\gamma+\zeta)}{\beta(E_\text{gr})+\gamma(E_\text{gr})+\zeta(E_\text{gr})}.\quad\qed\]\end{proof}

As we shall see in Section \ref{examples}, the choice of $\dot{c}_0=\mathrm{pr}_{\ker(\beta+\gamma+\zeta)}(E_\text{gr})$ seems to work in most---but not all---cases. Specifically, another choice of $\dot{c}_0$ is needed for the example in \ref{pathgeom}.

\subsection{Holonomy}
In the situation of Theorem \ref{stumps}, let $X\in\mathfrak{g}_{-\nu_0}<\mathfrak{j}_\Omega$. By Theorem 2.2 of \cite{Hammerl}, the Lie algebra of the restricted holonomy group at $(\exp(X),e)$, which we identify with $\exp(X)\in J_\Omega$, of the curvature tree grown from $\Omega$ is given by \begin{align*}\mathfrak{hol}^\circ_{\exp(X)}(J_\Omega\times_{K_\Omega}P,\MC{\Omega}) & =\mathrm{im}(\Omega)+[\psi(\mathfrak{j}_\Omega),\mathrm{im}(\Omega)] \\ & \quad +[\psi(\mathfrak{j}_\Omega),[\psi(\mathfrak{j}_\Omega),\mathrm{im}(\Omega)]]+\cdots \\ & =\mathrm{im}(\Omega)+[\mathfrak{g}_-+\mathfrak{k}_\Omega,\mathrm{im}(\Omega)] \\ & \quad +[\mathfrak{g}_-+\mathfrak{k}_\Omega,[\mathfrak{g}_-+\mathfrak{k}_\Omega,\mathrm{im}(\Omega)]]+\cdots \\ & =\mathrm{im}(\Omega)+[\mathfrak{g}_-,\mathrm{im}(\Omega)] \\ & \quad +[\mathfrak{g}_-,[\mathfrak{g}_-,\mathrm{im}(\Omega)]]+\cdots,\end{align*} since $\mathfrak{k}_\Omega$ is the Lie algebra of the stabilizer of $\Omega$ by definition, so since $\partial\Omega=0$, this tells us that \[\mathfrak{hol}^\circ_{\exp(X)}(J_\Omega\times_{K_\Omega}P,\MC{\Omega})\leq\sum_{i,j\geq 0}\mathfrak{g}_{\zeta-i\beta-j\gamma}.\] In particular, $\ker(\beta+\gamma+\zeta)\cap\mathfrak{hol}^\circ_{\exp(X)}(J_\Omega\times_{K_\Omega}P,\MC{\Omega})=\{0\}$. Because $J_\Omega/K_\Omega$ is simply connected, the holonomy group of the curvature tree grown from $\Omega$, and hence the holonomy group of the curvature tree with $o$ removed, is equal to its restricted holonomy group $\Hol^\circ_{\exp(X)}(\Omega):=\Hol^\circ_{\exp(X)}(J_\Omega\times_{K_\Omega}P,\MC{\Omega})$, the connected subgroup generated by $\mathfrak{hol}^\circ_{\exp(X)}(J_\Omega\times_{K_\Omega}P,\MC{\Omega})$.

Therefore, from the remarks in Section 6 of \cite{HolonomyPaper}, it follows that the holonomy group of the induced Cartan geometry on \[\mathscr{G}_{\Omega,t_0}:=c_{t_0\mathbb{Z}}\backslash((J_\Omega\times_{K_\Omega}P)\setminus\{q_{{}_P}^{-1}(o)\})\] at the point $c_{t_0\mathbb{Z}}\exp(X)\in\mathscr{G}_{\Omega,t_0}$ is given by the subgroup of $G$ generated by $\Hol^\circ_{\exp(X)}(\Omega)$ and $\exp(\psi(X))^{-1}c_{t_0\mathbb{Z}}\exp(\psi(X))$. Moreover, by the description above, $\Hol^\circ_{\exp(X)}(\Omega)$ is normalized by $\exp(\psi(X))$, and $c_{t_0\mathbb{Z}}$ normalizes it as well by Proposition 6.2 of \cite{HolonomyPaper}, so \[\exp(\psi(X))\Hol_{c_{t_0\mathbb{Z}}\exp(X)}(\mathscr{G}_{\Omega,t_0},\MC{\Omega})\exp(\psi(X))^{-1}=\Hol^\circ_{\exp(X)}(\Omega)c_{t_0\mathbb{Z}}.\]

The conjugacy class of the holonomy group is invariant under geometric isomorphism, so if two Cartan geometries have non-conjugate holonomy groups, then they are not isomorphic. Thus, in order to show that $(\mathscr{G}_{\Omega,t_0},\MC{\Omega})$ and $(\mathscr{G}_{\Omega,t_1},\MC{\Omega})$ are not isomorphic for $t_1>t_0>0$, we just need to show that there is no $g\in G$ such that \[g\Hol^\circ_{\exp(X)}(\Omega)c_{t_0\mathbb{Z}}g^{-1}=\Hol^\circ_{\exp(X)}(\Omega)c_{t_1\mathbb{Z}}.\] Indeed, if there were such a $g\in G$, then $gc_{t_0}g^{-1}\in\Hol^\circ_{\exp(X)}(\Omega)c_{t_1}$, but $\Hol^\circ_{\exp(X)}(\Omega)$ is unipotent and normalized by $c_\mathbb{R}$, so this would force $c_{t_0}=c_{t_1}$, which contradicts $t_0<t_1$. In other words, we have the following result.

\begin{proposition}\label{distinct} For $t_1>t_0>0$, the Cartan geometries $(\mathscr{G}_{\Omega,t_0},\MC{\Omega})$ and $(\mathscr{G}_{\Omega,t_1},\MC{\Omega})$ are not isomorphic.\end{proposition}

Thus, the above construction gives an infinite family of distinct non-flat regular normal Cartan geometries with essential automorphisms over a closed base manifold.

\section{Examples}\label{examples}
Here, we give a few representative examples. This is not a complete list; indeed, we suspect that our construction gives at least one example for each parabolic model geometry $(G,P)$ with $G$ simple of real rank at least $3$, though as noted in \ref{acstrat}, we have not yet proven this.

\subsection{Almost Grassmannian structures}\label{almostgrassmannian}
For $m>3$ and $k\leq m/2$, let $G=\PGL_m\mathbb{R}$ and \[P=\left\{\begin{pmatrix}A & B \\ 0 & C\end{pmatrix}\in G:\begin{matrix}A\in\GLin_k\mathbb{R}, C\in\GLin_{m-k}\mathbb{R}, \\ \text{and }B\in M_{k\times(m-k)}(\mathbb{R})\end{matrix}\right\}.\] In this case, regular normal Cartan geometries of type $(G,P)$ correspond to almost Grassmannian structures. These are split-real $|1|$-graded geometries, with \begin{align*}\mathfrak{g}_{-1} & =\left\{\begin{pmatrix}0 & 0 \\ X & 0\end{pmatrix}:X\in M_{(m-k)\times k}(\mathbb{R})\right\}, \\ \mathfrak{g}_0 & =\left\{\begin{pmatrix}X & 0 \\ 0 & Y\end{pmatrix}:X\in\mathfrak{gl}_k\mathbb{R},Y\in\mathfrak{gl}_{m-k}\mathbb{R}\right\},\text{ and} \\ \mathfrak{g}_1 & =\left\{\begin{pmatrix}0 & X \\ 0 & 0\end{pmatrix}:X\in M_{k\times(m-k)}(\mathbb{R})\right\}.\end{align*} The grading element is given by \[E_\text{gr}=\begin{pmatrix}\mathds{1} & 0 \\ 0 & 0\end{pmatrix}=\begin{pmatrix}\mathds{1} & 0 \\ 0 & 0\end{pmatrix}-\frac{k}{m}\begin{pmatrix}\mathds{1} & 0 \\ 0 & \mathds{1}\end{pmatrix}=\frac{1}{m}\begin{pmatrix}(m-k)\mathds{1} & 0 \\ 0 & -k\mathds{1}\end{pmatrix}\in\mathfrak{pgl}_m\mathbb{R}.\]
Thus, for $\theta:X\mapsto -X^\top$, $\mathfrak{c}=\mathfrak{a}$ the image in $\mathfrak{pgl}_m\mathbb{R}$ of the subalgebra of diagonal matrices in $\mathfrak{gl}_m\mathbb{R}$, $\varepsilon_i-\varepsilon_j\in\mathfrak{a}^\vee$ given by \[(\varepsilon_i-\varepsilon_j)\begin{pmatrix}d_1 & 0 & \cdots & 0 \\ 0 & d_2 & \cdots & 0 \\ \vdots & \vdots & \ddots & \vdots \\ 0 & 0 & \cdots & d_m\end{pmatrix}=d_i-d_j,\] $\Delta=\{\varepsilon_i-\varepsilon_j:i\neq j\}$, and $\Delta^+=\{\varepsilon_i-\varepsilon_j:i<j\}$, $(\theta,\mathfrak{c},\mathfrak{a},\Delta,\Delta^+)$ is a compatible (restricted) root system for $(G,P)$.

Since $G$ is split-real, we can just check Yamaguchi's list in \cite{Yamaguchi2} to get a lowest weight vector \[(\eta_{\alpha_k})_\kgf\wedge(\eta_{s_{\alpha_k}(\alpha_{k+1})})_\kgf\otimes\eta_{-s_{\alpha_k}s_{\alpha_{k+1}}(\varepsilon_1-\varepsilon_m)},\] which is a harmonic seed by Corollary \ref{seedcheck}. This takes the form $(\eta_{\varepsilon_2-\varepsilon_3})_\kgf\wedge(\eta_{\varepsilon_2-\varepsilon_4})_\kgf\otimes\eta_{\varepsilon_2-\varepsilon_1}$ for $k=2$ and $m=4$, $(\eta_{\varepsilon_1-\varepsilon_2})_\kgf\wedge(\eta_{\varepsilon_1-\varepsilon_3})_\kgf\otimes\eta_{\varepsilon_m-\varepsilon_2}$ for $k=1$ and $m>3$, and $(\eta_{\varepsilon_k-\varepsilon_{k+1}})_\kgf\wedge(\eta_{\varepsilon_k-\varepsilon_{k+2}})_\kgf\otimes\eta_{\varepsilon_m-\varepsilon_1}$ for $k>1$ and $m>4$.

For $k=2$ and $m=4$, our harmonic seed is of the form \[\Omega=(\eta_{\varepsilon_2-\varepsilon_3})_\kgf\wedge(\eta_{\varepsilon_2-\varepsilon_4})_\kgf\otimes\eta_{\varepsilon_2-\varepsilon_1}.\] Here, we have \[(\varepsilon_2-\varepsilon_3)+(\varepsilon_2-\varepsilon_4)+(\varepsilon_2-\varepsilon_1)=-\varepsilon_1+3\varepsilon_2-\varepsilon_3-\varepsilon_4=4\varepsilon_2,\] and \[(4\varepsilon_2)^\kgf=(-\varepsilon_1+3\varepsilon_2-\varepsilon_3-\varepsilon_4)^\kgf=\frac{1}{8}\begin{pmatrix}-1 & 0 & 0 & 0 \\ 0 & 3 & 0 & 0 \\ 0 & 0 & -1 & 0 \\ 0 & 0 & 0 & -1\end{pmatrix}=\frac{1}{8}\begin{pmatrix}0 & 0 & 0 & 0 \\ 0 & 4 & 0 & 0 \\ 0 & 0 & 0 & 0 \\ 0 & 0 & 0 & 0\end{pmatrix}.\] 
Notably, this shows that $(4\varepsilon_2)^\kgf\not\in\mathfrak{z}(\mathfrak{g}_0)$, since in this case $\mathfrak{z}(\mathfrak{g}_0)$ is just the span $\langle E_\text{gr}\rangle$ of $E_\text{gr}$, so $(4\varepsilon_2)^\kgf$ is not a scaling element. Since \[\mathfrak{g}_0^\text{ss}\cap\mathfrak{a}=\left\{\begin{pmatrix}r & 0 & 0 & 0 \\ 0 & -r & 0 & 0 \\ 0 & 0 & s & 0 \\ 0 & 0 & 0 & -s\end{pmatrix}:r,s\in\mathbb{R}\right\}\] is $2$-dimensional, we can apply Proposition \ref{bigg0case} to get $\dot{a}_0$. Explicitly, let us pick $\alpha=\varepsilon_1-\varepsilon_4$. Then, \[\mathrm{pr}_{\mathfrak{g}_0^\text{ss}\cap\mathfrak{a}}(4\varepsilon_2)^\kgf=\frac{1}{4}\begin{pmatrix}-1 & 0 & 0 & 0 \\ 0 & 1 & 0 & 0 \\ 0 & 0 & 0 & 0 \\ 0 & 0 & 0 & 0\end{pmatrix}\] and \[\mathrm{pr}_{\mathfrak{g}_0^\text{ss}\cap\mathfrak{a}}(\varepsilon_2-\varepsilon_3)^\kgf=\frac{1}{16}\begin{pmatrix}-1 & 0 & 0 & 0 \\ 0 & 1 & 0 & 0 \\ 0 & 0 & -1 & 0 \\ 0 & 0 & 0 & 1\end{pmatrix},\] so $\mathrm{pr}_{\mathfrak{g}_0^\text{ss}\cap\mathfrak{a}}(\varepsilon_2-\varepsilon_3)^\kgf\not\in\langle\mathrm{pr}_{\mathfrak{g}_0^\text{ss}\cap\mathfrak{a}}(4\varepsilon_2)^\kgf\rangle$, hence there is at least one $R\in\mathfrak{g}_0^\text{ss}\cap\mathfrak{a}$, namely $R=0$, such that $4\varepsilon_2(R)=-\kgf(4\varepsilon_2,\alpha)=0$ and $(\varepsilon_2-\varepsilon_3)(R)=-\kgf(\varepsilon_2-\varepsilon_3,\alpha)=0$. In other words, for $\nu_0=\varepsilon_2-\varepsilon_3$, \[\dot{a}_0=\alpha^\kgf+0\in\ker(4\varepsilon_2)\cap\ker(\nu_0).\] For $\dot{c}_0$, note that for $\varepsilon_i-\varepsilon_j\in\Delta^+(\mathfrak{p}_+)=\{\varepsilon_i-\varepsilon_j:1\leq i\leq 2<j\leq 4\}$, \[\kgf(4\varepsilon_2,\varepsilon_i-\varepsilon_j)=\left\{\begin{matrix}1/2\text{ if }i=2, \\ 0\text{ otherwise,}\end{matrix}\right.\] and \[\frac{\kgf(4\varepsilon_2,4\varepsilon_2)}{4\varepsilon_2(E_\text{gr})}=\frac{(3/2)}{2}=\frac{3}{4}>\frac{1}{2},\] so \[\dot{c}_0=\mathrm{pr}_{\ker(4\varepsilon_2)}(E_\text{gr})=\begin{pmatrix}1 & 0 & 0 & 0 \\ 0 & 1/3 & 0 & 0 \\ 0 & 0 & 0 & 0 \\ 0 & 0 & 0 & 0\end{pmatrix}=\frac{1}{3}\begin{pmatrix}2 & 0 & 0 & 0 \\ 0 & 0 & 0 & 0 \\ 0 & 0 & -1 & 0 \\ 0 & 0 & 0 & -1\end{pmatrix}\] satisfies $\nu(\dot{c}_0)>0$ for all $\nu\in\Delta^+(\mathfrak{p}_+)$ by Proposition \ref{shrink}. By Theorem \ref{stumps} and Proposition \ref{distinct}, we therefore get an infinite family of regular normal Cartan geometries of type $(\PGL_4\mathbb{R},P)$ with nonvanishing curvature and essential automorphisms over a closed base manifold.

Similarly, for $k=1$ and $m>3$, our harmonic seed is of the form \[\Omega=(\eta_{\varepsilon_1-\varepsilon_2})_\kgf\wedge(\eta_{\varepsilon_1-\varepsilon_3})_\kgf\otimes\eta_{\varepsilon_m-\varepsilon_2}.\] Note that \[(\varepsilon_1-\varepsilon_2)+(\varepsilon_1-\varepsilon_3)+(\varepsilon_m-\varepsilon_2)=2(\varepsilon_1-\varepsilon_2)-\varepsilon_3+\varepsilon_m,\] and $(2(\varepsilon_1-\varepsilon_2)-\varepsilon_3+\varepsilon_m)^\kgf\not\in\mathfrak{z}(\mathfrak{g}_0)=\langle E_\text{gr}\rangle$, so in particular it is not a scaling element and for each choice of $\alpha\in\Delta^+(\mathfrak{p}_+)=\{\varepsilon_1-\varepsilon_i:1<i\leq m\}$, we have a choice of $\nu_0\in\Delta^+(\mathfrak{p}_+)$ and $R$ as in Proposition \ref{bigg0case}. Moreover, \[\kgf(2(\varepsilon_1-\varepsilon_2)-\varepsilon_3+\varepsilon_m,\varepsilon_1-\varepsilon_i)=\left\{\begin{matrix}2/m \text{ if } i=2, \\ \frac{3}{2m} \text{ if } i=3, \\ 1/m \text{ if } 3<i<m, \\ \frac{1}{2m} \text{ if } i=m,\end{matrix}\right.\] and \[\frac{\kgf(2(\varepsilon_1-\varepsilon_2)-\varepsilon_3+\varepsilon_m,2(\varepsilon_1-\varepsilon_2)-\varepsilon_3+\varepsilon_m)}{(\varepsilon_1-\varepsilon_2)(E_\text{gr})+(\varepsilon_1-\varepsilon_3)(E_\text{gr})+(\varepsilon_m-\varepsilon_2)(E_\text{gr})}=\frac{2^2+2^2+1^2+1^2}{2m(1+1+0)}=\frac{5}{2m},\] so by Proposition \ref{shrink}, $\dot{c}_0=\mathrm{pr}_{\ker(2(\varepsilon_1-\varepsilon_2)-\varepsilon_3+\varepsilon_m)}(E_\text{gr})$ satisfies $(\varepsilon_1-\varepsilon_i)(\dot{c}_0)>0$ for all $i$. Thus, by Theorem \ref{stumps} and Proposition \ref{distinct}, we get infinitely many non-flat regular normal Cartan geometries of type $(\PGL_m\mathbb{R},P)$ with essential automorphisms on a closed base manifold.

Finally, for $k>1$ and $m>4$, our harmonic seed is of the form \[\Omega=(\eta_{\varepsilon_k-\varepsilon_{k+1}})_\kgf\wedge(\eta_{\varepsilon_k-\varepsilon_{k+2}})_\kgf\otimes\eta_{\varepsilon_m-\varepsilon_1}.\] The weight of $\Omega$ is then given by \[(\varepsilon_k-\varepsilon_{k+1})+(\varepsilon_k-\varepsilon_{k+2})+(\varepsilon_m-\varepsilon_1)=2\varepsilon_k-\varepsilon_1-\varepsilon_{k+1}-\varepsilon_{k+2}+\varepsilon_m,\] and $(2\varepsilon_k-\varepsilon_1-\varepsilon_{k+1}-\varepsilon_{k+2}+\varepsilon_m)^\kgf\not\in\mathfrak{z}(\mathfrak{g}_0)=\langle E_\text{gr}\rangle$, hence it is not a scaling element and for each choice of $\alpha\in\Delta^+(\mathfrak{p}_+)=\{\varepsilon_i-\varepsilon_j:1\leq i\leq k<j\leq m\}$, we have a choice of $\nu_0\in\Delta^+(\mathfrak{p}_+)$ and $R\in\mathfrak{g}_0^\text{ss}\cap\mathfrak{a}$ as in Proposition \ref{bigg0case}. For our choice of $\dot{c}_0$, \[\kgf(2\varepsilon_k-\varepsilon_1-\varepsilon_{k+1}-\varepsilon_{k+2}+\varepsilon_m,\varepsilon_i-\varepsilon_j)\leq\frac{3}{2m}\] and \[\frac{\kgf(2\varepsilon_k-\varepsilon_1-\varepsilon_{k+1}-\varepsilon_{k+2}+\varepsilon_m,2\varepsilon_k-\varepsilon_1-\varepsilon_{k+1}-\varepsilon_{k+2}+\varepsilon_m)}{(\varepsilon_k-\varepsilon_{k+1})(E_\text{gr})+(\varepsilon_k-\varepsilon_{k+2})(E_\text{gr})+(\varepsilon_m-\varepsilon_1)(E_\text{gr})}=\frac{4}{m},\] so $\dot{c}_0=\mathrm{pr}_{\ker(2\varepsilon_k-\varepsilon_1-\varepsilon_{k+1}-\varepsilon_{k+2}+\varepsilon_m)}(E_\text{gr})$ satisfies $\nu(\dot{c}_0)>0$ for all $\nu\in\Delta^+(\mathfrak{p}_+)$ by Proposition \ref{shrink}. Thus, we once again get infinitely many regular normal Cartan geometries of type $(\PGL_m\mathbb{R},P)$ with essential automorphisms and nonvanishing curvature on a closed base manifold.

\subsection{$(\PGL_4\mathbb{R},B)$ with $B$ a Borel subgroup}
Here, we again have $G=\PGL_4\mathbb{R}$, but this time we choose our parabolic subgroup to be the Borel subgroup $B$, corresponding to the upper triangular matrices. This is a split-real $|3|$-graded geometry, with $\mathfrak{g}_0=\mathfrak{c}=\mathfrak{a}$ corresponding to images of diagonal matrices in $\mathfrak{pgl}_4\mathbb{R}$ and $\mathfrak{g}_i=\langle E_{k\ell}:\ell-k=i\rangle$ for $i\neq 0$, where we use $E_{k\ell}$ to denote the image in $\mathfrak{pgl}_4\mathbb{R}$ of the matrix whose only nonzero entry is $1$ in the $k^\text{th}$ row and $\ell^\text{th}$ column. The grading element is given by \[E_\text{gr}=\begin{pmatrix}4 & 0 & 0 & 0 \\ 0 & 3 & 0 & 0 \\ 0 & 0 & 2 & 0 \\ 0 & 0 & 0 & 1\end{pmatrix}=\begin{pmatrix}3/2 & 0 & 0 & 0 \\ 0 & 1/2 & 0 & 0 \\ 0 & 0 & -1/2 & 0 \\ 0 & 0 & 0 & -3/2\end{pmatrix}\in\mathfrak{pgl}_4\mathbb{R},\] so the $(\theta,\mathfrak{c},\mathfrak{a},\Delta,\Delta^+)$ from the previous example is also a compatible (restricted) root system for $(G,B)$.

Again, because $G$ is split-real, we can just read off a lowest weight vector from Yamaguchi's list \cite{Yamaguchi2}. In this case, however, there are lowest weight vectors that are harmonic seeds (by Corollary \ref{seedcheck}) but whose duals via $\kgf$ are scaling elements.

For instance, $(\eta_{\varepsilon_1-\varepsilon_2})_\kgf\wedge(\eta_{\varepsilon_3-\varepsilon_4})_\kgf\otimes\eta_{\varepsilon_3-\varepsilon_2}$ is a lowest weight vector in $\ker(\square)_+$, but \[(\varepsilon_1-\varepsilon_2)+(\varepsilon_3-\varepsilon_4)+(\varepsilon_3-\varepsilon_2)=\varepsilon_1-2\varepsilon_2+2\varepsilon_3-\varepsilon_4\] and $(\varepsilon_1-2\varepsilon_2+2\varepsilon_3-\varepsilon_4)^\kgf\in\mathfrak{z}(\mathfrak{g}_0)$ is a scaling element, so the construction does not work for it.

This case also has lowest weight vectors whose duals via $\kgf$ are not scaling elements. As an example, consider \[\Omega=(\eta_{\varepsilon_2-\varepsilon_3})_\kgf\wedge(\eta_{\varepsilon_2-\varepsilon_4})_\kgf\otimes\eta_{\varepsilon_2-\varepsilon_1}.\] Here, the weight is \[(\varepsilon_2-\varepsilon_3)+(\varepsilon_2-\varepsilon_4)+(\varepsilon_2-\varepsilon_1)=-\varepsilon_1+3\varepsilon_2-\varepsilon_3-\varepsilon_4=4\varepsilon_2,\] and $\kgf(4\varepsilon_2,\varepsilon_1-\varepsilon_3)=0$, so $(4\varepsilon_2)^\kgf\in\mathfrak{z}(\mathfrak{g}_0)$ is not a scaling element. Using Proposition \ref{centralnotscaling}, we can pick $\alpha=\varepsilon_1-\varepsilon_3$ and $\nu_0=\varepsilon_2-\varepsilon_4$ to get \[\dot{a}_0=(\varepsilon_1-\varepsilon_3)^\kgf\in\ker(4\varepsilon_2)\cap\ker(\varepsilon_2-\varepsilon_4).\] Moreover, $\kgf(4\varepsilon_2,\varepsilon_i-\varepsilon_j)\leq 1/2$ and \[\frac{\kgf(4\varepsilon_2,4\varepsilon_2)}{4\varepsilon_2(E_\text{gr})}=\frac{(3/2)}{2}=\frac{3}{4}>\frac{1}{2},\] so \[\dot{c}_0=\mathrm{pr}_{\ker(4\varepsilon_2)}(E_\text{gr})=\frac{1}{3}\begin{pmatrix}5 & 0 & 0 & 0 \\ 0 & 0 & 0 & 0 \\ 0 & 0 & -1 & 0 \\ 0 & 0 & 0 & -4\end{pmatrix}\] satisfies $\nu(\dot{c}_0)>0$ for all $\nu\in\Delta^+(\mathfrak{p}_+)$. Thus, we get a corresponding infinite family of regular normal Cartan geometries of type $(\PGL_4\mathbb{R},B)$ with essential automorphisms and nonvanishing curvature over a closed base manifold by Theorem \ref{stumps} and Proposition \ref{distinct}.

\subsection{Path geometries}\label{pathgeom}
Again, we take our Lie group $G=\PGL_m\mathbb{R}$ (with $m>3$), but with \[P=\left\{\begin{pmatrix}r & t & p \\ 0 & s & q \\ 0 & 0 & A\end{pmatrix}:~\begin{matrix}p^\top\!,q^\top\in\mathbb{R}^{m-2},~r,s\in\mathbb{R}^\times, \\ t\in\mathbb{R},\text{ and }A\in\GLin_{m-2}\mathbb{R}\end{matrix}\right\}.\] In this case, regular normal Cartan geometries of type $(G,P)$ correspond to path geometries. These are split-real $|2|$-graded geometries, which we can grade using \[E_\text{gr}=\begin{pmatrix}2 & 0 & 0 \\ 0 & 1 & 0 \\ 0 & 0 & 0\end{pmatrix}=\frac{1}{m}\begin{pmatrix}2m-3 & 0 & 0 \\ 0 & m-3 & 0 \\ 0 & 0 & -3\mathds{1}\end{pmatrix}\in\mathfrak{pgl}_m\mathbb{R}\] as the grading element. Taking $(\theta,\mathfrak{c},\mathfrak{a},\Delta,\Delta^+)$ as in \ref{almostgrassmannian}, we get a compatible (restricted) root system for $(G,P)$.

Because $G$ is, again, split-real, we can use Yamaguchi's list \cite{Yamaguchi2} to get a lowest weight vector in $\ker(\square)_+$. For example, we can use \[\Omega=(\eta_{\varepsilon_2-\varepsilon_3})_\kgf\wedge(\eta_{\varepsilon_1-\varepsilon_3})_\kgf\otimes\eta_{\varepsilon_m-\varepsilon_3}\in\ker(\square)_+.\] The corresponding weight is \[(\varepsilon_2-\varepsilon_3)+(\varepsilon_1-\varepsilon_3)+(\varepsilon_m-\varepsilon_3)=\varepsilon_1+\varepsilon_2+\varepsilon_m-3\varepsilon_3,\] and $(\varepsilon_1+\varepsilon_2+\varepsilon_m-3\varepsilon_3)^\kgf\not\in\mathfrak{z}(\mathfrak{g}_0)$ is not a scaling element. For $\alpha=\varepsilon_1-\varepsilon_m$ and $\nu_0=\varepsilon_2-\varepsilon_3$, \[\kgf(\alpha,\nu_0)=\kgf(\varepsilon_1+\varepsilon_2+\varepsilon_m-3\varepsilon_3,\alpha)=0,\] so $\dot{a}_0=\alpha^\kgf+0\in\ker(\varepsilon_1+\varepsilon_2+\varepsilon_m-3\varepsilon_3)\cap\ker(\nu_0)$.

When picking $\dot{c}_0$, note that \[\kgf(\varepsilon_1+\varepsilon_2+\varepsilon_m-3\varepsilon_3,\varepsilon_2-\varepsilon_3)=\frac{1+3}{2m}=\frac{2}{m}\] and \[\frac{\kgf(\varepsilon_1+\varepsilon_2+\varepsilon_m-3\varepsilon_3,\varepsilon_1+\varepsilon_2+\varepsilon_m-3\varepsilon_3)}{(\varepsilon_1+\varepsilon_2+\varepsilon_m-3\varepsilon_3)(E_\text{gr})}=\frac{(6/m)}{3}=\frac{2}{m}\not>\frac{2}{m},\] so $\mathrm{pr}_{\ker(\varepsilon_1+\varepsilon_2+\varepsilon_m-3\varepsilon_3)}(E_\text{gr})$ is not a valid choice of $\dot{c}_0$ in this case. However, \[\dot{c}_0=E_\text{gr}-3\begin{pmatrix}0 & \cdots & 0 & 0 \\ \vdots & \ddots & \vdots & \vdots \\ 0 & \cdots & 0 & 0 \\ 0 & \cdots & 0 & 1\end{pmatrix}\] satisfies \[(\varepsilon_1+\varepsilon_2+\varepsilon_m-3\varepsilon_3)(\dot{c}_0)=2+1-3-0=0,\] $(\varepsilon_1-\varepsilon_2)(\dot{c}_0)=2-1=1>0$, and \[(\varepsilon_2-\varepsilon_i)(\dot{c}_0)=\left\{\begin{matrix}1 & \text{if }2<i<m \\ 4 & \text{if }i=m,\end{matrix}\right.\] so $\nu(\dot{c}_0)>0$ for all $\nu\in\Delta^+(\mathfrak{p}_+)$, hence $\dot{c}_0$ is a valid choice, giving us an infinite family of regular normal Cartan geometries of type $(\PGL_m\mathbb{R},P)$ with essential automorphisms and nonvanishing curvature over a closed base manifold by Theorem \ref{stumps} and Proposition \ref{distinct}.

\subsection{Quaternionic contact geometries of signature $(m,n)$}
Throughout, we denote by $\hat{i}$, $\hat{j}$, and $\hat{k}$ the usual elements of $\mathbb{H}$, with hats to help distinguish them from indices. For $n\geq m>1$, let $G$ be the subgroup of $\GLin_{m+n+2}\mathbb{H}$ preserving the $\mathbb{R}$-valued quadratic form \[v\mapsto \overline{v_1}v_{m+n+2}+\overline{v_{m+n+2}}v_1+\left(\sum_{\ell=1}^m |v_{\ell+1}|^2\right)-\left(\sum_{\ell=1}^n |v_{m+1+\ell}|^2\right)\] on $\mathbb{H}^{m+n+2}$, and let $P$ be the closed subgroup of $G$ preserving the subspace $\mathbb{H}e_1\subseteq\mathbb{H}^{m+n+2}$, where $e_\ell$ denotes the element whose $\ell^\text{th}$ entry is $1$ and with all other entries equal to $0$. At the level of Lie algebras, \[\mathfrak{g}=\left\{\begin{bmatrix}a & p & q & c \\ v & R_1 & \bar{S}^\top & -\bar{p}^\top \\ w & S & R_2 & \bar{q}^\top \\ z & -\bar{v}^\top & \bar{w}^\top & -\bar{a}\end{bmatrix}:\begin{matrix}a\in\mathbb{H},~ v,p^\top\in\mathbb{H}^m,~ w,q^\top\in\mathbb{H}^n, \\ R_1\in\mathfrak{sp}(m),~ R_2\in\mathfrak{sp}(n), \\ S\in M_{n\times m}(\mathbb{H}),\text{ and }c,z\in\mathrm{Im}(\mathbb{H})\end{matrix}\right\}\] and \[\mathfrak{p}=\left\{\begin{bmatrix}a & p & q & c \\ 0 & R_1 & \bar{S}^\top & -\bar{p}^\top \\ 0 & S & R_2 & \bar{q}^\top \\ 0 & 0 & 0 & -\bar{a}\end{bmatrix}:\begin{matrix}a\in\mathbb{H},~ p^\top\in\mathbb{H}^m,~ q^\top\in\mathbb{H}^n, \\ R_1\in\mathfrak{sp}(m),~ R_2\in\mathfrak{sp}(n), \\ S\in M_{n\times m}(\mathbb{H}),\text{ and }c\in\mathrm{Im}(\mathbb{H})\end{matrix}\right\}.\] Here, regular normal Cartan geometries of type $(G,P)$ correspond to quaternionic contact structrures of signature $(m,n)$. These are $|2|$-graded, but not split-real; the grading element is given by \[E_\text{gr}=\begin{bmatrix}1 & 0 & 0 & 0 \\ 0 & 0 & 0 & 0 \\ 0 & 0 & 0 & 0 \\ 0 & 0 & 0 & -1\end{bmatrix}.\]
Let us define $\theta:X\mapsto-\bar{X}^\top$, \[\mathfrak{c}=\left\{\begin{bmatrix}a & 0 & 0 & 0 & 0 \\ 0 & L_1 & D & 0 & 0 \\ 0 & D & L_1 & 0 & 0 \\ 0 & 0 & 0 & L_2 & 0 \\ 0 & 0 & 0 & 0 & -\bar{a}\end{bmatrix}:~\begin{matrix}a\in\mathbb{C}<\mathbb{H},~ D,\hat{i}L_1\in M_{m\times m}(\mathbb{R}),~ \\ \text{and }\hat{i}L_2\in M_{(n-m)\times(n-m)}(\mathbb{R}), \\ \text{with }D,L_1,\text{ and }L_2\text{ diagonal}\end{matrix}\right\},\] where we identify $\mathbb{C}$ with the $\mathbb{R}$-subalgebra $\{a+b\hat{i}:a,b\in\mathbb{R}\}$ of $\mathbb{H}$, and \[\mathfrak{a}=\left\{\begin{bmatrix}a & 0 & 0 & 0 & 0 \\ 0 & 0 & D & 0 & 0 \\ 0 & D & 0 & 0 & 0 \\ 0 & 0 & 0 & 0 & 0 \\ 0 & 0 & 0 & 0 & -a\end{bmatrix}:~\begin{matrix}a\in\mathbb{R}\text{ and }D\in M_{m\times m}(\mathbb{R}) \\ \text{with }D\text{ diagonal}\end{matrix}\right\}<\mathfrak{c},\] so that $E_\text{gr}\in\mathfrak{a}$. Further, for \[D=\begin{bmatrix}d_1 & 0 & \cdots & 0 \\ 0 & d_2 & \cdots & 0 \\ \vdots & \vdots & \ddots & \vdots \\ 0 & 0 & \cdots & d_m\end{bmatrix},\] let us define $\varepsilon_\ell\in\mathfrak{a}^\vee$ by \[\varepsilon_\ell\left(\left[\begin{smallmatrix}a & 0 & 0 & 0 & 0 \\ 0 & 0 & D & 0 & 0 \\ 0 & D & 0 & 0 & 0 \\ 0 & 0 & 0 & 0 & 0 \\ 0 & 0 & 0 & 0 & -a\end{smallmatrix}\right]\right)=\left\{\begin{matrix}a & \text{if }\ell=0, \\ d_\ell & \text{if }1\leq\ell\leq m\end{matrix}\right.\] Then, for $\Delta=\{\pm(\varepsilon_{\ell_1}-\varepsilon_{\ell_2}),\pm(\varepsilon_{\ell_1}+\varepsilon_{\ell_2}):0\leq\ell_1<\ell_2\leq m\}\cup\{\pm 2\varepsilon_\ell,\pm\varepsilon_\ell:0\leq\ell\leq m\}$ and $\Delta^+=\{\varepsilon_{\ell_1}-\varepsilon_{\ell_2},\varepsilon_{\ell_1}+\varepsilon_{\ell_2}:0\leq\ell_1<\ell_2\leq m\}\cup\{2\varepsilon_\ell,\varepsilon_\ell:0\leq\ell\leq m\}$ when $m<n$, and $\Delta=\{\pm(\varepsilon_{\ell_1}-\varepsilon_{\ell_2}),\pm(\varepsilon_{\ell_1}+\varepsilon_{\ell_2}):0\leq\ell_1<\ell_2\leq m\}\cup\{\pm 2\varepsilon_\ell:0\leq\ell\leq m\}$ and $\Delta^+=\{\varepsilon_{\ell_1}-\varepsilon_{\ell_2},\varepsilon_{\ell_1}+\varepsilon_{\ell_2}:0\leq\ell_1<\ell_2\leq m\}\cup\{2\varepsilon_\ell:0\leq\ell\leq m\}$ when $m=n$, $(\theta,\mathfrak{c},\mathfrak{a},\Delta,\Delta^+)$ is a compatible restricted root system for $(G,P)$.

Let us pick $\beta=\gamma=\varepsilon_0-\varepsilon_1\in\Delta^+(\mathfrak{p}_+)$ and $\zeta=-2\varepsilon_1\in -\Delta^+$, and write \[\eta_{\varepsilon_0-\varepsilon_1,q}:=\begin{bmatrix}0 & qe_1^\top & qe_1^\top & 0 \\ 0 & 0 & 0 & -\bar{q}e_1 \\ 0 & 0 & 0 & \bar{q}e_1 \\ 0 & 0 & 0 & 0\end{bmatrix}\] for $q\in\mathbb{H}$ and \[\eta_{-2\varepsilon_1,q}:=\begin{bmatrix}0 & 0 & 0 & 0 \\ 0 & qe_1e_1^\top & qe_1e_1^\top & 0 \\ 0 & -qe_1e_1^\top & -qe_1e_1^\top & 0 \\ 0 & 0 & 0 & 0\end{bmatrix}\] for $q\in\mathrm{Im}(\mathbb{H})$. Then, \begin{align*}\Omega & =(\eta_{\beta,1})_\kgf\wedge(\eta_{\gamma,\hat{i}})_\kgf\otimes\eta_{\zeta,\hat{j}}+(\eta_{\beta,\hat{k}})_\kgf\wedge(\eta_{\gamma,\hat{j}})_\kgf\otimes\eta_{\zeta,\hat{j}} \\ & \quad+(\eta_{\beta,1})_\kgf\wedge(\eta_{\gamma,\hat{j}})_\kgf\otimes\eta_{\zeta,\hat{i}}-(\eta_{\beta,\hat{k}})_\kgf\wedge(\eta_{\gamma,\hat{i}})_\kgf\otimes\eta_{\zeta,\hat{i}}\end{align*} is a lowest weight vector of positive homogeneity in $\ker(\square)$. Since $\zeta\not\in\Delta^+(\mathfrak{p}_+)$ and $\zeta\neq -(\varepsilon_0-\varepsilon_1)$, $\Omega$ is a harmonic seed by Theorem \ref{main}.

Moreover, $\beta+\gamma+\zeta=2\varepsilon_0-4\varepsilon_1$, and since $(2\varepsilon_0-4\varepsilon_1)^\kgf\not\in\mathfrak{z}(\mathfrak{g}_0)=\langle E_\text{gr}\rangle$, it is not a scaling element, so for each $\alpha\in\Delta^+(\mathfrak{p}_+)$, we can find $\nu_0\in\Delta^+(\mathfrak{p}_+)$ and $R\in\mathfrak{g}_0^\text{ss}\cap\mathfrak{a}$ such that $\dot{a}_0=\alpha^\kgf+R\in\ker(\beta+\gamma+\zeta)\cap\ker(\nu_0)$ by Proposition \ref{bigg0case}. For $\dot{c}_0$, note that $\kgf(\varepsilon_0\pm\varepsilon_1,2\varepsilon_0-4\varepsilon_1)=\frac{2\mp 4}{8(m+n+3)}$, \[\kgf(\varepsilon_0\pm\varepsilon_\ell,2\varepsilon_0-4\varepsilon_1)=\kgf(\varepsilon_0,2\varepsilon_0-4\varepsilon_1)=\frac{2}{8(m+n+3)}\] for $\ell>1$, and \[\frac{\kgf(2\varepsilon_0-4\varepsilon_1,2\varepsilon_0-4\varepsilon_1)}{(2\varepsilon_0-4\varepsilon_1)(E_\text{gr})}=\frac{10}{8(m+n+3)},\] so by Proposition \ref{shrink}, $\dot{c}_0=\mathrm{pr}_{\ker(2\varepsilon_0-4\varepsilon_1)}(E_\text{gr})$ satisfies $\nu(\dot{c}_0)>0$ for each $\nu\in\Delta^+(\mathfrak{p}_+)$. Thus, by Theorem \ref{stumps} and Proposition \ref{distinct}, this gives an infinite family of regular normal Cartan geometries of type $(G,P)$ with nonvanishing curvature and essential automorphisms over a closed manifold.

\end{document}